\documentclass[12pt, leqno]{amsart}
\usepackage[utf8]{inputenc}
\usepackage[english]{babel}
\usepackage{amstext}
\usepackage{amsthm}
\usepackage{amsmath}
\usepackage{amssymb}
\usepackage{graphicx}
\usepackage{mathtools}
\usepackage{fancyvrb}
\usepackage[margin=0.99in]{geometry}

\usepackage[unicode=true,pdfusetitle,
 bookmarks=true,bookmarksnumbered=false,bookmarksopen=false,
 breaklinks=false,pdfborder={0 0 1},backref=false,colorlinks=true]
 {hyperref}
 \usepackage{mathrsfs}
 \usepackage{bm}

\usepackage[show]{ed}

\hypersetup{
 linkcolor=red,urlcolor=red,citecolor=red}

\makeatletter

\AtBeginDocument{}
\AtBeginDocument{}

\numberwithin{equation}{section}
\numberwithin{figure}{section}

\theoremstyle{plain}
\newtheorem{thm}{\protect\theoremname}[section]

\theoremstyle{definition}

\theoremstyle{plain}
\newtheorem{lem}[thm]{\protect\lemmaname}

\theoremstyle{plain}
\newtheorem{cor}[thm]{\protect\corname}

\theoremstyle{plain}
\newtheorem{prop}[thm]{\protect\propositionname}

\theoremstyle{plain}
\newtheorem*{claim*}{\protect\claimname}

\theoremstyle{remark}

\newtheorem{remark}[thm]{\protect\remarkname}

\theoremstyle{plain}
\newtheorem{assum}[thm]{Assumption}

\def\R{\mathbb R}

\def\eps{\varepsilon}

\def\I{\operatorname{I}}

\makeatother

\providecommand{\definitionname}{Definition}
\providecommand{\lemmaname}{Lemma}
\providecommand{\propositionname}{Proposition}
\providecommand{\theoremname}{Theorem}
\providecommand{\remarkname}{Remark}
\providecommand{\claimname}{Claim}
\providecommand{\corname}{Corollary}

\theoremstyle{plain}
\newtheorem{claim}[thm]{Claim}

\title[Scattering for NLS under damped strong trapping]{Scattering for defocusing cubic NLS under locally damped strong trapping}
\author[D. Lafontaine]{David Lafontaine}
\address{CNRS and Institut de Mathématiques de Toulouse ; UMR5219, Université de Toulouse ; CNRS, UPS IMT, F-31062 Toulouse Cedex 9 (France)}
\email{david.lafontaine@math.univ-toulouse.fr}

\author[B. Shakarov]{Boris Shakarov}
\address{Institut de Mathématiques de Toulouse ; UMR5219, Université de Toulouse ; CNRS, UPS IMT, F-31062 Toulouse Cedex 9 (France)}
\email{boris.shakarov@math.univ-toulouse.fr}
\thanks{The work of B. S. is 
  partially supported by CIMI ANR-11-LABX-0040 and the 
  ANR project NQG ANR-23-CE40-0005}


\date{\today}

\begin{document}

\begin{abstract}
We are interested in the scattering problem for the cubic 3D nonlinear defocusing Schr\"odinger equation with variable coefficients. Previous scattering results for such problems address only the cases with constant coefficients or assume strong variants of the non-trapping condition, stating that all the trajectories of the  Hamiltonian flow associated with the operator are escaping to infinity. In contrast, we consider the most general setting, where strong trapping, such as stable closed geodesics, may occur, but we introduce a compactly supported damping term
 localized in the trapping region, to explore how damping can mitigate the effects of trapping.

In addition to the challenges posed by the trapped trajectories, notably the loss of smoothing and of scale-invariant Strichartz estimates, difficulties arise from the damping itself, particularly since the energy is not, a priori, bounded. For 
$H^{1+\epsilon}$ initial data -- chosen because the local-in-time theory is a priori no better than for 3D unbounded manifolds, where local well-posedness of strong $H^1$ solutions is unavailable -- we establish global existence and scattering in $H^{s}$ for any $0 \leq s <1$ in positive times,  the inability to reach $H^1$ 
being related to the loss of smoothing due to trapping.
 
\end{abstract}

\maketitle

\VerbatimFootnotes 

\section{Introduction}
We are interested in the following cubic defocusing nonlinear Schr\"odinger equation with a variable-coefficients Laplacian in divergence form and a compactly supported damping by a potential
\begin{equation} \label{eq:DNLS}
\begin{cases}
    i\partial_t u  + \Delta_G u + iau = |u|^2 u, \\
    u(0) = u_0 \in H^s(\R^3).
\end{cases}
\end{equation}
Here $\Delta_G$ denotes
$$
\Delta_G u:= \operatorname{div}(G \nabla u),
$$
with $G\in C^{\infty}(\R^{3\times 3}, \mathbb R)$ symmetric and uniformly positive-definite, in the sense that
$$
\exists c>0,  \hspace{0.3cm} \forall x \in \R^3,  \hspace{0.3cm} \forall \xi \in \R^3, \hspace{0.3cm} G(x)\xi \cdot \xi \geq c |\xi|^2,
$$
and $G = \I$ outside a compact set
$$
G - \I \in C^\infty_c(\R^{3\times 3}, \mathbb R). 
$$ 
In addition, the damping potential $a$ is non-negative and compactly supported
$$
a \in C^\infty_c(\mathbb R^3, \mathbb R_+),
$$
and in what follows, it will typically be active where $G-\I \neq 0$. 
The following quantities, respectively the total mass and energy, are conserved in the case $a=0$ and therefore will play a crucial role:
\begin{equation*}
    M[u(t)] := \int |u(t,x)|^2 dx, \hspace{0.4cm}E[u(t)] := \frac{1}{2} \int G(x) \nabla u(t,x) \cdot \nabla \bar u(t,x) \, dx + \frac{1}{4} \int |u(t,x)|^4 \, dx.
\end{equation*}
We are interested in the forward global-in-time behavior of solutions, more specifically, in showing scattering to linear solutions.

To motivate the problem, let us briefly review the related problems when $a=0$. When $G = \I$ and $a = 0$, equation (\ref{eq:DNLS}) have a rich history that we will not attempt to 
fully review here. In particular, 
the conservation laws together with the local theory give global-well-posedness when the data is in $H^1(\mathbb R^3)$, and it is known since the work of Ginibre and Velo \cite{GV85scat} that solutions scatter in $H^1$ to linear solutions, in the sense that there exists $u_\pm \in H^1(\mathbb R^3)$ so that
$$
\Vert u(t) - e^{it\Delta} u_\pm \Vert_{H^1(\mathbb R^3)} \to 0 \hspace{0.3cm}\text{ as }\hspace{0.3cm} t \to \pm\infty.
$$
Whereas we are not aware of any scattering result in the case $G \neq \I$ and $a=0$, scattering for $a=0$ has been investigated in at least two related inhomogeneous situations: the case where $G = \I$, but with the equation posed outside an obstacle (i.e. in $\mathbb R^3 \backslash \Theta$, with $\Theta$ compact with smooth boundary) with Dirichlet boundary conditions \cite{PV1, PV2, IvPl10, Ab15, KiViXi16a, KiViXi16b, BiTeQi22}, and the case where $G = \I$, but a perturbation $Vu$ by a potential is added \cite{BaVi16, Ho16, La16}. In both cases, scattering results rely crucially on strong \emph{non-trapping} conditions. The most general version of this condition states that all the Hamiltonian trajectories (these are parametrized curves in $\mathbb R_x^3 \times \mathbb R_\xi^3$) associated with the principal symbol of the operator having the role of Laplacian exit any compact of space in finite time. In the case of $G=\I$, these trajectories project in space to straight lines; outside an obstacle, reflections have to be taken into account and these are the so-called generalized broken bi-characteristics. In the case of  general $-\Delta_G$, the trajectory from $(x_0, \xi_0)$ is the solutions of the Hamilton equation
\begin{equation} \label{eq:Ham}
\dot x(t) =  \nabla_\xi(G(x) \xi \cdot \xi) = 2 G(x) \xi, \hspace{0.5cm} \dot \xi(t) = - \nabla_{x} (G(x) \xi \cdot \xi) = - \sum_{1\leq i, j \leq 3} \nabla G_{i,j}(x) \xi_i \xi_j,
\end{equation}
with $(x(0), \xi(0)) = (x_0, \xi_0)$.
Such a non-trapping assumption and stronger repulsivity variant discussed momentarily, plays a crucial role in the two main ingredients used to show scattering: global Strichartz estimates, and non-concentration estimates on the nonlinear solutions. Concerning the former, global-in-time Strichartz estimates typically hold in non-trapping situations. Indeed, at least in the case without boundaries, semiclassical Strichartz estimates (i.e., on time intervals of size $\sim h$ for data localized in frequencies $\sim h^{-1}$) hold without geometric assumption due to the semiclassical finite speed of propagation \cite{BuGeTz04}. On the other hand, the non-trapping condition implies global-in-time smoothing (see e.g. \cite[\S2.3]{BuGeTz04b} for a proof from resolvent estimates, themselves tracing back to \cite{LaPh89, MeSj82, Va89, VaZw00, Bu02}), which can in turn be used to deduce global-in-time Strichartz estimates from semiclassical ones \cite{Iv10, StTa02} (see also \cite{Bu03}, \cite{BuGuHa10}). For the problem with boundaries, such global Strichartz estimates involve in general loss of derivatives, however, this difficulty can be overcome using the global smoothing close to the (non-trapping) obstacle \cite{PV1, PV2}. Non-concentration estimates, taking the form of Morawetz estimates such as originating from the work of Lin and Strauss \cite{LinStrauss}, are the second crucial ingredient in a typical proof of scattering. These can for example be used in their modern interaction form \cite{PV1, CoKeStTaTa04}, to show that the solution is in a global-in-time space-time Lebesgue space (which is then used, by interpolation with energy conservation arguments, to show that it is in a scale-invariant space-time Lebesgue space at the level we are interested in, which in turn implies scattering), or to rule-out compact flow solutions in a concentration-compactness/rigidity type argument by contradiction in the now classical Kenig-Merle scheme originating from \cite{keMe06}. Such Morawetz estimates rely on even stronger non-trapping conditions, namely \emph{repulsivity conditions}, to ensure that the term arising from the perturbation in the computation has the right sign. For an obstacle, this is the assumption that it is star-shaped, i.e. $x\cdot n(x) \leq 0$ at the boundary; for a potential, it typically takes the form $x\cdot \nabla V \geq 0$ (or $(x\cdot \nabla V)_-$ small).

In contrast, we are interested in the most general situation without any non-trapping assumption, hence with the Hamilton flow (\ref{eq:Ham}) associated with $-\Delta_G$ having possibly the strongest (i.e., the most stable) possible trapped trajectories, but we introduce a damping term $iau$, active where the trapping takes place, intending to understand how such a damping can mitigate the effects of trapping.  Hence the damping $a$ will verify the following control condition:
\begin{equation} \label{eq:control}
\operatorname{supp}(G- \I) \subset \big\{ a>0 \big\}.
\end{equation}
\begin{figure}
\begin{center}
\includegraphics[scale=0.8]{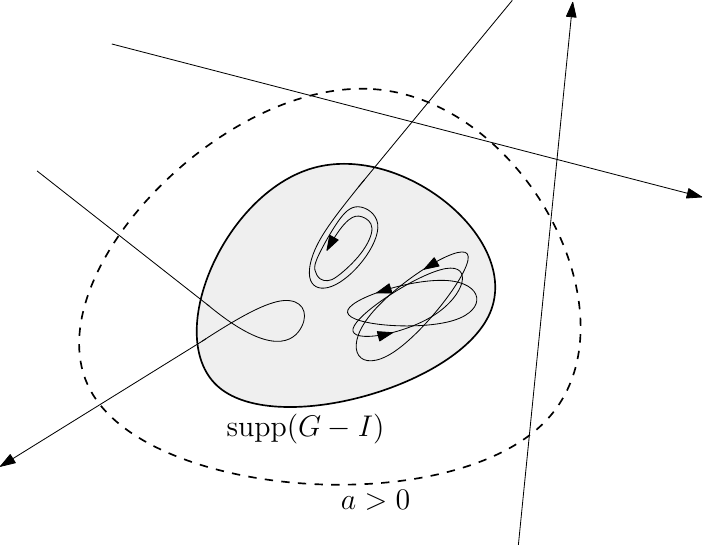}
\caption{Schematic representation of the control condition (\ref{eq:control}) and of typical Hamilton trajectories.} 
\end{center}
\end{figure}
The overall idea is that, at least at the $ L^2$-level, the damping should allow recovering \emph{some} global-in-time integrability in the region of trapping, where Morawetz-type arguments break down. On the other hand, numerous challenges are introduced \emph{both} by the trapping and the damping, namely:
\begin{enumerate}
\item\label{d:1} \emph{Loss of smoothing:} one of the most spectacular effects of the existence of stable trapped trajectories is that no smoothing effect, even with a loss of derivative, can be true for the associated linear Schrödinger operator (see e.g. \cite[Remark 4.2]{Bu04}). Adding a damping term can permit to recover some global-in-time integrability at the $L^2$-level for linear solutions $u_L$, in the form that $\Vert u_L \Vert_{L^2(\mathbb R_+, B(0, R))} \lesssim_R \Vert u_0 \Vert_{L^2(\R^3)}$ under the assumption that any trapped trajectory intersects $\{ a > 0 \}$ \cite{AlKh07}; however damping by a potential doesn't permit to recover any smoothing (see also \cite{AlKhRo17, KhRo17}).
\item\label{d:2} \emph{Loss in local-in-time Strichartz estimates:} Consequently, smoothing cannot be used to obtain even local-in-time scale-invariant Strichartz estimates from semiclassical estimates. Therefore, we have a priori no better than the same local-in-time Strichartz estimates than in a general manifold, which involves a loss of $\frac 1p$ derivatives \cite{BuGeTz04}. As a consequence, the local-well-posedness theory is a priori no better than in a general 3D manifold, for which the usual contraction principle  gives well-posedness of strong $C([0,T], H^s)$ solutions only for $s>1$. Recall however that \cite{BuGeTz04} managed to show global existence and uniqueness of weak $H^1$ solutions in a bounded manifold, but we lack the appropriate compactness tools to replicate their argument in our unbounded setting.
\item\label{d:3} \emph{No a priori uniform-in-time bound on the energy:} whereas the mass is non-increasing thanks to the damping, the derivative of the energy is a priori not signed, hence no a priori uniform-in-time bound on the energy is at hand.
\item\label{d:4} \emph{Negative times:} Finally, let us mention that the linear operator is not bounded uniformly in negatives times in $L^2(\R^3)$, hence it makes the use of the Duhamel formula involving the retarded evolution to show scattering from space-time Lebesgue bounds challenging.
\end{enumerate}

We are now ready to introduce our main result. It shows that (\ref{eq:DNLS}) is locally well-posed in $H^{1+\epsilon}$, and under the control condition (\ref{eq:control}), solutions are uniformly bounded forward in time in the energy space, and scatter with a loss of one-half derivative with respect to the free case. This constitutes the first result of scattering \emph{both} under strong geometrical trapping, and for non-constant damping, for any nonlinear dispersive equation.
\begin{thm}\label{thm:ScatDef} 
For any $\epsilon>0$, equation (\ref{eq:DNLS}) is locally well-posed in $H^{1+\epsilon}(\mathbb R^3)$. In addition, under the control condition (\ref{eq:control}), for any $u_0 \in H^{1+\epsilon}(\mathbb R^3)$, the associated forward maximal solution is global in positive times, verifies $\Vert u \Vert_{L^\infty(\mathbb R_+, H^1(\mathbb R^3))} \leq C \Vert u_0 \Vert_{H^1(\mathbb R^3)}$, and scatters in $H^{1-}(\R^3)$ to a free linear wave, in the sense that there exists $u_+ \in H^{1}(\mathbb R^3)$ so that
$$
\forall s\in[0, 1), \hspace{0.3cm}\text{ }\Vert u(t) - e^{it\Delta} u_+ \Vert_{H^{s}(\mathbb R^3)} \to 0 \hspace{0.3cm} \text{ as } \hspace{0.3cm} t \to + \infty.
$$
\end{thm}

The fact that the well-posedness holds in $H^{1+\epsilon}$ is related to (\ref{d:2}) above.
 On the other hand, we see the inability to reach scattering in $H^1$ as related to the loss of smoothing in the linear flow (\ref{d:1}) due to trapping --
in particular, if $G$ induces a non-trapping Hamiltonian flow, both well-posedness and scattering hold in $H^1$, as discuted at the end of this introduction and shown in Appendix \ref{app:non_trap}. 

We now give some ideas of the proof. The first preliminary step, carried out in \S\ref{secPrelim}, is to show local well-posedness together with a conditional global-well-posedness result. Strichartz estimates for the undamped linear flow from \cite{BuGeTz04} give local well-posedness in $H^{1+\epsilon}$. We will show in a later key step that solutions are uniformly bounded forward in time in the energy space, but for $H^{1+\epsilon}$ well-posed solutions,
this is a-priori not enough to obtain global existence. However, from an adaptation of the argument used in \cite{BuGeTz04} to show the existence of a unique global weak $H^1$ solution, we can show that a uniform bound on the energy implies global forward existence of $H^{1+\epsilon}$ well-posed solutions.

Scattering in $H^{1-}$ will be obtained by interpolation from scattering in $H^{\frac 12}$. Scattering in $H^{\frac 12}$ in turn will be a consequence of a local energy decay estimate (without gain of regularity) for solutions of (\ref{eq:DNLS}) once forward global well-posedness is established, namely, we will show that:
\begin{equation} \label{eq:intro_energy_decay}
\Vert u \Vert_{L^2(\mathbb R_+, H^1(B(0, R)))} \lesssim_R \Vert u_0 \Vert_{H^1(\mathbb R^3)}.
\end{equation}
We show both the above and a uniform bound on the energy, hence global well-posedness, in one go. This will be the main result of \S\ref{s:nrj}. The idea is the following. 
We first remark that, from the energy law,  
\begin{equation} \label{eq:intro_energy_bound}
    E[u(t)] \lesssim E[u(0)] + \Vert u \Vert_{L^2([0, t], L^2(B(0, R)))},
\end{equation}
for $R>0$ large enough, hence it suffices to show decay of the local mass, $\Vert u \Vert_{L^2(\mathbb R_+, L^2(B(0, R)))}<\infty$, to obtain a uniform energy bound. In order to do so, we use a Morawetz-type estimate, that makes precisely appear the localized mass as the term coming from the linear part of the equation. More precisely, we perform a Morawetz  computation (on the maximal time of existence) with linear weight $\langle x \rangle$, perturbatively with respect to $G = \I$, hence with an unsigned error term localized on $\operatorname{supp}(G-\I)$. The localized mass term arising in the computation is then 
$$
\lambda(t) := \int_0^t \int \langle x\rangle^{-7}|u|^2 dxdt.
$$
We control the energy by $\lambda$, and, thanks to the energy law and the control condition, we are able to control the error terms \emph{up to lower order terms in $\lambda$}. We show in such a way that $\lambda$ verifies an inequality of the form
$$
\lambda(t)^2 - b\lambda(t) - c \leq 0,
$$
and is therefore bounded. The uniform energy bound and global existence follow, and coming back to the Morawetz computation and plugging the information that $\sup_t \lambda(t) < \infty$ we obtain local energy decay (\ref{eq:intro_energy_decay}).

Once (\ref{eq:intro_energy_bound}) is established, the second ingredient is a bilinear (or interaction) Morawetz estimate in the spirit of \cite{PV1, CoKeStTaTa04}, showing that solutions to (\ref{eq:DNLS}) are in $L^4(\mathbb R_+, L^4(\R^3))$, obtained in \S\ref{s:bil}. The idea is to perform the computation perturbatively with respect to $G= \I$, hence with error terms again localized in $\operatorname{supp}(G-I)$. But as  $\operatorname{supp}(G-I)$ is compact, we can control these terms thanks to the local energy decay (\ref{eq:intro_energy_bound}), and the global space-time bound follows. 

To conclude from a global space-time bound, one typically takes advantage of the Duhamel formula and global Strichartz estimates to construct $u_+$. However, in our setting, this direct approach is hopeless, as it would require global estimates backward in time for the linear group. One cannot consider the damping $iau$ as a perturbation and put it in the source term either, as the linear group would now be $e^{it\Delta_G}$, for which no global estimate is at hand due to trapping. To overcome these difficulties, we cut the solution in two parts: a part $\chi u$ localized close to $\operatorname{supp}a$, and a part $(1-\chi) u$  away from it. From local energy decay, one can show that $\chi u \to 0$ in any $H^s(\R^3)$, $s<1$. On the other hand, $(1-\chi) u$ verifies a nonlinear Schr\"odinger equation with the free Laplacian $-\Delta$, and one can show scattering in $H^{\frac 12}$ of this part by expressing it thanks to Duhamel with the free linear group $e^{it\Delta}$. This involves control of the nonlinearity $u|u|^2$, which comes from the $L^4 L^4$-bound, and control of the commutant $[\Delta, \chi]u$ arising in the equation, which follows thanks to local energy decay. Once scattering in $H^{\frac 12}$ is established, it can be improved to $H^{1-}$ by interpolation with the uniform bound on energy; and the proof is completed. This final step is carried out in \S\ref{secEnd}.

Remark that scattering of the part $(1-\chi)u$ cannot be improved, before interpolation, beyond $H^{\frac 12}(\R^3)$ with our method, as it would require a bound on $[\Delta, \chi]u$ at a higher level of regularity, which is not at hand. 
One cannot, either, improve the result a posteriori by cutting again the solution in the spirit of \cite{PV1}, as this approach requires local smoothing for the linear group, which does not hold in our case due to trapping. This is the reason why we relate the inability to reach $H^1$ in our scattering result to the loss of smoothing. 

To conclude this introduction, we now discuss other related problems, generalizations presented in Appendices, and open questions. Observe that the problem we are interested in can be seen as related to the stabilization and controllability of the equation. Stabilization is typically expected under the geometric control condition, stating that every Hamiltonian trajectory intersects the control region $\{ a > 0 \}$. In unbounded settings, stabilization results therefore typically assume that the control is active at infinity:  $\{ a > 0 \}\subset \mathbb R^d \backslash B(0, R)$, $R\gg1$. We refer for example to \cite{DeGeLe06, RoZh10, Lau10,  CaCoWe20, YaNiChe21, BaMa23, stab24}, and references therein, for the reader interested in the stabilization problem. In contrast, in our case (\ref{eq:control}), the control is active only locally near the perturbation: it constitutes the simplest example of an \emph{exterior} control condition, stating that any trajectory either go to infinity, or intersects the control region.
The case of a constant damping $a\equiv1$ (with constant coefficients $G=\I$), enjoys the best of both worlds, and, by conjugating with an exponential, one can show exponential scattering in $H^1$ \cite{Ta19}.

Remark that our result is new even in the case $G = \I$. In this case, and actually, under a general mild trapping assumption (i.e. for $G$  such that $e^{it\Delta_G}$ verifies global, scale-invariant Strichartz estimates and a local-energy decay estimate, see Assumption \ref{ass:mid_trap}) which holds in particular  in the non-trapping case, we can show scattering up to $H^1(\mathbb R^3)$ -- still under the control condition (\ref{eq:control}). This is done in Appendix \ref{app:non_trap}.

Finally, observe that the nonlinear parts of our arguments work just as well in the case of the equation posed outside an obstacle. However, the best general local Strichartz estimates (i.e., without any geometric conditions) for the linear Schr\"odinger equation in a setting with boundaries known at the moment are not enough to obtain a well-posedness theory such as presented in \S\ref{secPrelim}, hence the result obtained is conditional. This is illustrated in Appendix \ref{app:obst}.

We finish this introduction by stating some open questions related to our problem. First, we don't know weather scattering in $H^{1-}$ is sharp. Going to $H^1$, or giving a counter-example, would likely involve a better understanding of the related \emph{linear} damped flow. In particular, could one show better small time Strichartz estimates for the damped flow than for the undamped one, under the exterior control condition, or is there a counter-example? Related to this question, could one show that the loss of smoothing 
 for the damped flow in the results of \cite{AlKh07,AlKhRo17, KhRo17} is sharp? This is expected, but no counter-example seems to be known. Coming back to the non-linear equation, it is natural to expect an analogous result, under a natural mass-energy threshold, for the focusing equation. Showing such a result would likely involve profile decompositions in the spirit of \cite{keMe06}, which will be challenging to carry out in our setting, in particular due to the non-self-adjointness of the problem, and the fact that the linear operator is not bounded in negative times. Finally, observe that even in the non-trapping case, without damping, all the known scattering results involve a stronger repulsivity assumption mentioned earlier. In particular, it is natural to expect scattering for the undamped equation in any non-trapping geometry, but this is not known -- the restriction to repulsive geometries is mostly due to the rigidity of the Morawetz computations in its various avatars used to show scattering. In the spirit of this paper, we expect an even stronger conjecture to hold: scattering (at least in $H^{1-}$) under the sole exterior control condition (any trajectory either meets the control condition or goes to infinity). This last question is however still far from reach.

\subsection*{Notations}
We write $a \lesssim b$ to indicate that there exists a universal constant $C>0$ such that $a \leq Cb$. We define $H^{1-}(\mathbb R^3)$ as $H^{1-}(\mathbb R^3) := \cap_{0\leq s<1} H^s(\mathbb R^3)$ with $H^0(\mathbb R^3) := L^2(\mathbb R^3)$.

\section{Preliminaries}\label{secPrelim}

\subsection{Strichartz Estimates for the undamped linear Schrödinger flow}

We will say that the couple $(p,q)$ is admissible whenever
\begin{equation}\label{eq:StrichS}
2\leq p, q \leq \infty , \hspace{0.3cm} \frac 2 p + \frac 3 q = \frac 3 2 .
\end{equation}
From \cite{BuGeTz04}, we have local-in-time Strichartz estimates for the \emph{undamped} linear flow $e^{ i t \Delta_G}$ with a loss of $\frac 1p$ derivatives:
\begin{prop} \label{prop:strich}
For any $T >0$, any $s \geq 0$ and any admissible couple $(p,q)$, there exists $C>0$ such that, for any $u_0 \in H^{\frac 1p + s}(\R^3)$ and any 
$f \in L^1((0,T), H^{\frac 1p + s}(\mathbb R^3))$,
\begin{equation}\label{eq:StrLossHom}
\| e^{ i t \Delta_G} u_0 \|_{L^p((0,T),W^{s,q}(\R^3))} \leq C \| u_0 \|_{H^{ \frac 1p + s}(\R^3)},
\end{equation}
and
\begin{equation}\label{eq:NonHomDG}
    \left\| \int_0^t e^{i(t- \tau) \Delta_G} f \right\|_{L^p((0,T),W^{s,q}(\R^3))} \leq C \| f \|_{L^1((0,T), H^{\frac 1p + s}(\R^3))}.
\end{equation}
\end{prop}

\begin{proof}
    These estimates are shown in \cite[Theorem 1]{BuGeTz04}. Albeit \cite[Theorem 1]{BuGeTz04} is stated for the Laplace-Beltrami operator $-\Delta_g$ in a compact manifold $(\mathcal M, g)$, the proof applies equally well in our setting -- note in particular that a generalization to the Laplace operator associated to a Riemmanian metric in 
    $\mathbb R^d$ is given as \cite[Appendix A.1]{BuGeTz04}, which is very close to our setting. 
    
    We briefly sketch the proof of \cite[Theorem 1]{BuGeTz04} for the interested reader, then explain the (minimal) changes in our setting. The authors work for frequency-localized data $\psi(-h^2 \Delta_g) u_0$, and show, thanks to a WKB approximation on the kernel of the associated operator, that the semiclassical Schr\"odinger flow verifies the dispersive estimate $|e^{it\Delta_g}\psi(-h^2 \Delta_g) u_0 (x)| \lesssim t^{-\frac d2}\Vert u_0 \Vert_{L^1}$ in times $\lesssim h$ -- this is the key \cite[Lemma 2.5]{BuGeTz04}. By the now classical \cite{KeTa98}, this implies Strichartz estimates for $e^{it\Delta_g}\psi(-h^2 \Delta_g)$ (including endpoint) in time intervals of size $\sim h$, and hence for $e^{it\Delta_g}\psi(-h^2 \Delta_g)$ in time intervals of size $\sim 1$ but with a loss $h^{-1/p}$, which corresponds to the derivative loss $H^{\frac 1p}$ after re-summation of all frequencies.
    
    The proof of \cite[Theorem 1]{BuGeTz04} therefore uses two main ingredients: a frequency localization $\psi(-h^2 \Delta_g)$ and its properties when acting on spaces $L^p$, and a WKB construction used to show a dispersive estimate $L^1\to L^\infty$ in times $\sim 1$ for the semiclassical Schr\"odinger
    flow $e^{ith\Delta_g}$. 
    
    In the case of the Laplace operator associated to a Riemmanian metric in 
    $\mathbb R^d$ presented in \cite[Appendix A.1]{BuGeTz04}, the main difference is the different status of the $L^p$ spaces in $\mathbb R^d$: for this reason, the authors show the dispersive estimate for a localization in frequencies $\Psi(hD)$, that is, $|e^{it\Delta_g}\Psi(hD) u_0 (x)| \lesssim t^{-\frac d2}\Vert u_0 \Vert_{L^1}$ in times $\lesssim h$ -- this is \cite[Lemma A.3]{BuGeTz04}, the analogue of the key \cite[Lemma 2.5]{BuGeTz04}. The localization $\Psi(hD)$ is a good approximation to $\psi(-h^2 \Delta_g)$ by \cite[Proposition A.1, Corollary A.2]{BuGeTz04}, and the Strichartz estimates follow as in the compact case thanks to this supplementary ingredient. The proof of the dispersive estimate \cite[Lemma A.3]{BuGeTz04} itself is almost verbatim the same as in the compact case, to the difference that the solutions to the eikonal and transport equations are now defined globally in space for small times.
    
    In our case, the Strichartz estimates follow almost verbatim as in \cite[Appendix A.1]{BuGeTz04} recapped above. Indeed, $-\Delta_G$ satisfies the assumptions of \cite[Proposition A.1]{BuGeTz04}, hence \cite[Proposition A.1, Corollary A.2]{BuGeTz04} apply. On the other hand, the same WKB method applies (without the need to work in coordinates patch)
    in the  same way to show the analogue of \cite[Lemma 2.5]{BuGeTz04}, 
    by solving the eikonal and transport equations that now write
    $$
    \partial_s \phi + G\nabla \phi \cdot \nabla \phi = 0,
    $$
    $$
    \partial_s a_0 + 2 G \nabla \phi \cdot \nabla a_0 + \Delta_G (\phi) a_0 = 0,
    $$
    $$
    \partial_s a_j+ 2 G \nabla \phi \cdot \nabla a_j + \Delta_G (\phi) a_j = - \Delta_G (a_{j-1}) \hspace{0.5cm}j\geq 1,
    $$
    which are the analogue to \cite[(2.19)-(2.20)-(2.21)]{BuGeTz04},
    and the end of the proof follows as at the end of \cite[Appendix A.1]{BuGeTz04}.
\end{proof}

\begin{remark}[Global Strichartz estimates with dramatic loss for the damped flow] \label{rk:dram}
    Recall that, in the non-trapping case, the frequency-localized Strichartz estimates of \cite{BuGeTz04} can be combined with local smoothing to obtain global Strichartz estimates without loss of derivative \cite{Iv10, StTa02}. This strategy doesn't apply in our case as because of trapping, local smoothing is lost. 
    For the damped linear flow however (which we will denote $S(t)$ for the sake of this Remark), even if smoothing is not at hand, the damping term permits to recover some global-in-time integrability at the $L^2$-level: $\Vert S(t) u_0 \Vert_{L^2(\mathbb R_+, B(0, R))} \lesssim_R \Vert u_0 \Vert_{L^2(\R^3)}$ under the assumption that any trapped trajectory intersects $\{ a > 0 \}$ \cite{AlKh07}. This can be combined with local Strichartz estimates for $S(t)$ (which hold from Proposition \ref{prop:strich} by absorption by viewing  $iau$ as a source term) to obtain global Strichartz estimates for the damped linear flow $S(t)$ with a loss of $\frac 1p + \frac 12$ derivatives:
    $$
\Vert S(t) u_0 \Vert_{L^p(\R_+, L^q(\R^3))} \lesssim \Vert u_0 \Vert_{H^{\frac 1p + \frac 12}(\R ^3)},
    $$
    such a loss however seems too dramatic to be useful for our purposes, and we don't use such estimates in this paper.
\end{remark}
\subsection{Local, and conditionally global, well-posedness theory}
We start by showing local well-posedness of solutions in $H^{1+\eps}(\R^3)$ for any $\eps >0$ by the classical contraction principle. 
\begin{prop}\label{prpLocExsCubic} For any $\eps >0$, the followings holds.
    There exists $p>2$ so that for any $u_0 \in H^{1+\eps} (\R^3)$ there exists $T>0$ and a unique solution  $u \in C([0,T],H^{1 + \eps}(\R^3)) \cap L^p([0, T], L^\infty(\R^3))$ to \eqref{eq:DNLS}. In addition
    \begin{enumerate}
    \item If $\Vert u_0 \Vert_{H^{1+\eps}(\R^3)}$ is bounded, then $T$ is bounded below.
    \item The map $u_0 \in H^{1+\eps}(\R^3) \mapsto u \in C([0,T],H^{1+\eps}(\R^3))$ is continuous.
    \item If $u_0 \in H^m(\R^3)$ for $m>{1+\eps}$, then $u \in C([0, T], H^m(\R^3))$.  
    \end{enumerate}
    In particular, from (1), if $T_{\rm max} >0$ denotes the maximal forward time of existence, either $T_{\rm max} = +\infty$ or
    \begin{equation}\label{eqBUAlt}
        \lim_{t \to T_{\rm max}^-}\|  u (t) \|_{H^{{1+\eps}}(\R^3)} = \infty.
    \end{equation}
\end{prop}

\begin{proof}
The proof is similar to \cite[Proposition 3.1]{BuGeTz04} dealing with $iau$ as an additional source term.
We fix $\eps>0$ and any $(p,q)$ satisfying $p,q>2$ and
\begin{equation}\label{eqPQEq}
    s := 1 + \eps > \frac 32 - \frac 1p, \quad \quad \frac 2p + \frac 3q = \frac 32.
\end{equation}
Observe that \eqref{eqPQEq} implies that
\begin{equation*}
    \sigma := s - \frac 1p> \frac 3q,
\end{equation*}
which yields the Sobolev embedding 
\begin{equation} \label{eq:sob_lwp}
W^{\sigma, q}(\R^3) \hookrightarrow L^\infty(\R^3).
\end{equation} 
We define
\begin{equation*}
     X_{T} := L^\infty((0,T), H^{1 + \eps}(\R^3)) \cap L^p((0,T), W^{\sigma, q}(\R^3)),
 \end{equation*}
 equipped with 
 \begin{equation*}
     \| u \|_{X_T} := \| u \|_{L^\infty((0,T), H^{1 + \eps}(\R^3))} + \| u \|_{ L^p((0,T), W^{s, q}(\R^3))},
 \end{equation*}
 where $T>0$ is to be chosen later.  For any $t \in (0,T)$ and any $u_0 \in H^{1 + \eps}(\R^3)$, we define the map
 \begin{equation*}
     \phi(u_0,t)(u) = e^{it\Delta_G} u_0 - i \int_0^t e^{i(t-s)\Delta_G} (|u(s)|^{2} u(s) - i a u(s) )  ds.
 \end{equation*}
 We will show that for $T$ small enough, the map $\phi$ is a contraction on a ball centered at the origin to itself. 
 Indeed, by Strichartz estimates \eqref{eq:StrLossHom}, \eqref{eq:NonHomDG}, H\"older's inequality, Sobolev embedding (\ref{eq:sob_lwp}), and the fractional Leibniz rule, we have
 \begin{equation*}
 \begin{aligned}
     \| \phi(u_0, .) (u) \|_{X_T} &\leq C \| u_0 \|_{H^{1 + \eps}(\R^3)} + C  \| | u|^{2} u - i a u \|_{L^1((0,T), H^{1 + \eps}(\R^3))} \\
     & \leq C \| u_0 \|_{H^{1 + \eps}(\R^3)}  \\ 
     &\hspace{0.2cm}+ C \| u \|_{L^\infty((0,T),H^{1 + \eps}(\R^3))} (T\| a \|_{W^{1 +\eps,\infty}(\R^3)} + T^{\frac{p - 2}{p}} \| u \|^{2}_{L^p((0,T), L^\infty(\R^3))}) \\ 
     &\leq C \| u_0 \|_{H^{1 + \eps}(\R^3)}  + C_a(T +  T^{\frac{p- 2}{p}})(\| u \|_{X_T} + \| u \|_{X_T}^{3}).
      \end{aligned}
 \end{equation*}
Since $p - 2 >0$, there exists $T>0$ small enough so that $\phi$ send the ball of $X_T$, $B:= B_{X_T}\big(0, \frac C2 \Vert u_0 \Vert_{H^{1+\epsilon}(\R ^3)}\big)$, to itself.
In addition, we have similarly 
\begin{equation*}
    \|  \phi(u) - \phi(v) \|_{X_T} \lesssim (T +  T^{\frac{p - 2}{p}})  (1+\| u \|_{L^\infty((0,T),H^{1 + \eps}(\R^3)) }^{2}+ \| v \|_{L^\infty((0,T),H^{1 +\eps}(\R^3)) }^2)\|  u -v \|_{X_T},
\end{equation*}
which implies that the map $\phi$ is a contraction on $B$ for $T>0$ small enough. Thus $\phi$ admits a fixed point, which is a solution to \eqref{eq:DNLS}. Uniqueness and properties (1)--(2)--(3) are shown similarly. 
\end{proof}

 \begin{remark}
    When $G = \I$, the proof of the local well-posedness is a direct consequence of Kato's method, as in e.g. \cite[Section $4.4$]{Ca03}. In particular, we can choose $\eps = 0$.
\end{remark}

In the sequel of the paper, we will show that solutions are uniformly bounded forward in time in the energy space. As we handle $H^{1+\epsilon}$ well-posed solutions given by Proposition \ref{prpLocExsCubic}, this is a-priori not enough to obtain global existence from the local well-posedness theory alone. But we can show that such a bound on the energy indeed implies global existence:  
\begin{prop} \label{prop:cond_glob}
 Assume that the following holds: there exists a continuous function $C: \mathbb R_+ \to \mathbb R_+^*$ so that,  for any $\epsilon >0$ and $u_0 \in H^{1+\epsilon}(\R^3)$, if  $u$ is the solution to (\ref{eq:DNLS}) on $[0, T]$ with data $u_0$, then $\sup_{[0, T]}\Vert u \Vert_{H^1(\mathbb R^3)} \leq C(\Vert u_0 \Vert_{H^1(\mathbb R^3)})$. Then, the maximal solutions to (\ref{eq:DNLS}) are global forward in time. 
\end{prop}

\begin{proof}
Denote $s := 1 + \epsilon$. In view of the local well-posedness theory given by Proposition \ref{prpLocExsCubic}, it suffices to show that $\Vert u(t) \Vert_{H^s(\R^3)}$ is bounded in bounded (positive) times.

\textbf{Step 1 -- smooth solutions.} We begin by dealing with smooth solutions.
    The proof is contained in \cite[\S 3.3]{BuGeTz04}, dealing with the damping $iau$ as an additional ``non-linearity''. We sketch the argument. As already observed, \cite{BuGeTz04} applies to our setting to show frequency localized Strichartz estimates for $e^{it\Delta_G}$. From these estimates, one can show the analogue of \cite[Lemma 3.6]{BuGeTz04} to our setting: namely,
    there exists $C>0$ so that for any $a<b$ with $b-a < 1$ and $u \in C([a, b], H^1(\mathbb R^3))$ a solution to 
    $$
i \partial_t u + \Delta_G u = - i au + |u|^2 u,
    $$
    one has
\begin{equation}\label{eq:3.11}
\Vert u_h \Vert_{L^2([a,b], L^6(\R^3))} \leq C h^{\frac 12}
\Vert  u_h\Vert_{L^2([a,b],H^1(\mathbb R^3))} + Ch^{\frac 12 + \epsilon} \Lambda(\Vert u \Vert_{L^\infty([a,b], H^1(\mathbb R^3))}),
\end{equation}
        for some $\epsilon > 0$, where $\Lambda : \mathbb R_+ \mapsto \mathbb R_+$ is non-decreasing, 
    and $u_h$ denotes a  localization of $u$ to frequencies $\sim h^{-1}$ ($0<h<1$).
Indeed, an inspection of the proof of \cite[Lemma 3.6]{BuGeTz04} reveals that in addition to the frequency localized Strichartz estimates, the only other ingredient is that the source term $f:= - i au + |u|^2 u$ ($= G'(|u|^2)u$ in \cite{BuGeTz04}) verifies
    $$
    |\nabla f| \leq C \langle u \rangle^{2m - 2} |\nabla u|
    $$
    for some $m< \frac 5 2$, which is verified in our case with $m = 2$. Next, in the same way as \cite[p.589-590]{BuGeTz04}, by a well-chosen summation, it follows that any  solution which is uniformly bounded in $H^1(\mathbb R^3)$ verifies
\begin{equation} \label{eq:GB}
\Vert u \Vert_{L^2([0,T], L^\infty(\R^3))} \leq C \Big( (T \log(2 + \Vert u \Vert_{L^2([0,T], H^s(\R^3))}))^{\frac 12} + 1\Big).
\end{equation}
    But by the Duhamel formula together with standard nonlinear estimates
    $$
    \Vert u \Vert_{L^\infty([0, T], H^s(\R^3))} \leq \Vert u_0 \Vert_{H^s(\R^3)} + C \int_0^T (\Vert u(t) \Vert_{L^\infty(\R^3)}^2 + 1)\Vert u(t) \Vert_{H^s(\R^3)} dt.
    $$
    If $u$ is a smooth solution, the above combined with (\ref{eq:GB}) shows by a Gronwall argument that $\Vert u(t) \Vert_{H^s(\R^3)}$ is bounded in bounded times. It follows that any smooth, $H^s(\R^3)$ solution is global and in $L^\infty_{\rm loc}H^s$. For future reference, note that we obtain a bound on $\Vert u \Vert_{L^\infty([0, T], H^s(\R^3))}$ which depends only, and continuously, on $T$ and $\Vert u \Vert_{L^\infty([0, T], H^1(\mathbb R^3))}$ (non-decreasingly in both).

\textbf{Step 2 -- Yudovitch argument.} To apply the above to a non-smooth solution, we will use a refinement of the Yudovitch argument \cite{Yu67} used to show the uniqueness of $L^\infty H^1$ solutions in 
\cite[p.591 - 592]{BuGeTz04}. Namely:
\begin{claim} The following holds.
\begin{enumerate}
\item[(i)]
For any $M>0$, any $T>0$, and any $\epsilon >0$, there exists $\delta>0$ so that, if $u$ and $\widetilde u$ are solutions to the integral equations
\begin{equation}\label{eq:duha}
u(t) = e^{it\Delta_G} u_0 - i \int_0^t e^{i(t-s)\Delta_G} (|u(s)|^{2} u(s) - i a u(s) )  ds,
\end{equation}
\begin{equation}
\widetilde u(t) = e^{it\Delta_G} \widetilde {u_0} - i \int_0^t e^{i(t-s)\Delta_G} (|\widetilde u(s)|^{2} \widetilde u(s) - i a \widetilde u(s) )  ds,
\end{equation}
defined on $[0, T]$ and so that 
$$
\Vert u \Vert_{L^\infty ((0, T) H^1(\mathbb R^3))} \leq M, \hspace{0.5cm} \Vert \widetilde u \Vert_{L^\infty ((0, T) H^1(\mathbb R^3))} \leq M,
$$
and
$$
\Vert \widetilde u(0) - u(0) \Vert_{L^2(\R^3)} \leq \delta,
$$ 
then
$$
\sup_{t \in [0, T]} \Vert \widetilde u(t) - u(t) \Vert_{L^2(\R^3)} \leq \epsilon.
$$
\item[(ii)] For any $T>0$, solutions to (\ref{eq:duha}) on $[0, T]$ are unique subject to the condition $u \in L^\infty([0, T], H^1(\mathbb R^3))$.
\end{enumerate}
\end{claim}
Point (ii) follows from almost the same proof as in \cite{BuGeTz04}, while for (i), the supplementary ingredient is to track the dependency on the initial data. We now give the details.
To show (i), we will show that there exists $\tau_0>0$ depending only on $M$ so that for any $\epsilon>0$ there is a $\delta>0$ so that
$$
\Vert \widetilde u(0) - u(0) \Vert_{L^2(\R^3)} \leq \delta,
\implies
\sup_{t \in [0, \tau_0]} \Vert \widetilde u(t) - u(t) \Vert_{L^2(\R^3)} \leq \epsilon.
$$
As $\tau_0>0$ depends only on $M>0$, the result follows by iterating on ${\sim \tau_0^{-1}T}$ sub-intervals. From (\ref{eq:3.11}) it follows as in \cite[(3.28) p. 591]{BuGeTz04} that, if $\Vert u \Vert_{L^\infty ((0, T) H^1(\mathbb R^3))} \leq M$, we have for any $0\leq t \leq T$, with a constant depending only on $M>0$, for any $p\geq 6$
\begin{equation} \label{eq:Yudo}
\Vert u \Vert_{L^2((0, t), L^p(\R^3))} \leq C (\sqrt{tp}+1).
\end{equation}
Introduce
$$
g(t) := \Vert \widetilde u(t) - u(t) \Vert^2_{L^2(\R^3)} = \Vert e^{-it\Delta_G} (\widetilde u(t) - u(t) )\Vert^2_{L^2(\R^3)}.
$$
Now, compute
\begin{align*}
\partial_t g &\leq 2 \operatorname{Im}\langle|u(t)|^2u(t) - |\widetilde u(t)|^2\widetilde u(t), u(t) - \widetilde u(t)\rangle_{L^2} \\
&\leq C \int (|u(t)|^2 + | \widetilde u (t)|^2 ) |u(t) - \widetilde u (t)|^2 \\
&\leq C (\Vert u(t) \Vert^2_{L^{2p}(\R^3)} + \Vert \widetilde u(t) \Vert^2_{L^{2p}(\R^3)})\Vert u(t) - \widetilde u (t) \Vert^2_{L^{2 \bar p}(\R^3)},
\end{align*}
with $p<\infty$ big and $\bar p = \frac{p}{p-1}$
(observe that this is the same as \cite[p. 591]{BuGeTz04}, but with an inequality in the first line instead of the equality, the damping  $iau$ introducing only a dissipative term).
From there, it follows as in \cite[p. 591]{BuGeTz04}, interpolating the $L^{2 \bar p}$-norm between the $L^2$-norm and the $L^6$-norm, which is bounded by Sobolev embedding, that
$$
\partial_t g \leq C (\Vert u(t) \Vert^2_{L^{2p}(\R^3)} + \Vert \widetilde u(t) \Vert^2_{L^{2p}(\R^3)}) g(t)^{1-\frac{3}{2p}},
$$
hence, integrating
$$
pg(t)^{\frac{3}{2p}} \leq C (\Vert u \Vert^2_{L^2((0, t)L^{2p}(\R^3))} + \Vert \widetilde u \Vert^2_{L^2((0, t)L^{2p}(\R^3))})+ pg(0)^{\frac{3}{2p}}
$$
and from (\ref{eq:Yudo})
$$
p g(t)^{\frac{3}{2p}} \leq C (tp + 1) + p g(0)^{\frac{3}{2p}}, 
$$
from which 
\begin{align*}
g(t) &\leq \Big[C (t + \frac 1 p) + g(0)^{\frac{3}{2p}}\Big]^{\frac{2p}{3}} \\
&\leq \Big[ 2 C(t + \frac 1 p) \Big]^{\frac{2p}{3}} + 2^{\frac{2p}{3}} g(0),
\end{align*}
(we used the inequality $(a+b)^q \leq 2^q(a^q + b^q)$ for $a, b \geq 0$). Now, first fix $\tau_0>0$ small enough so that $2C \tau_0 \leq \frac 12$. It then follows that, for any $t \in [0, \tau_0]$
$$
\Big[2 C(t + \frac 1 p) \Big]^{\frac{2p}{3}} \leq \Big[\frac 12 + \frac {2C} {p} \Big]^{\frac{2p}{3}},
$$
and the right-hand side goes to zero as $p \to \infty$. Let us fix $p \gg 6$ big enough so that
$$
\Big[\frac 12 + \frac C p \Big]^{\frac{2p}{3}} \leq \frac \epsilon 2.
$$
With such a $p$ being fixed, we now take $\delta >0$ small enough so that
$$
 2^{\frac{2p}{3}} \delta^{2} \leq \frac \epsilon 2.
$$
Part (i) of the claim follows. Part (ii) is shown similarly with $g(0) = 0$, exactly as in \cite{BuGeTz04}.

\textbf{Step 3 -- Conclusion.}
Let now $u_0 \in H^s(\R^3)$, $s= 1+\epsilon$, and $u \in C([0, T_{\rm max}), H^s(\R^3))$ the associated maximal forward solution. We approximate $u_0$ in $H^s(\R^3)$ by a sequence of functions $(u_0^n)_{n\geq 1}$ in $H^m(\R^3)$, with $m$ big enough. Let $u^n$ be the associated solutions. From the first step, they are global, and in $L^\infty H^1$ and $L^\infty_{\rm loc}H^s$, both uniformly in $n$. From the claim, point (i), for any $T>0$, $u^n$ is Cauchy in $C([0, T], L^2(\R^3))$. It follows that there exists a function $v$ so that
$$
\forall T>0, \hspace{0.5cm}\Vert u^n - v \Vert_{C([0, T], L^2(\R^3))} \to 0.
$$
From the above and the uniform $L^\infty H^1$ bound on $u^n$, it follows that $v \in L^\infty H^1$. Thus, by interpolation, we get
$$
\forall T>0, \; \forall 0 \leq \sigma < 1, \hspace{0.5cm}\Vert u^n - v \Vert_{C([0, T], H^\sigma(\R^3))} \to 0.
$$
Using the above and the uniform $ L_{\rm loc}^\infty H^s$ bound on $u^n$, we get $v \in L_{\rm loc}^\infty H^s$. Interpolating one last time, we get
$$
\forall T>0, \; \forall 0 \leq \sigma < s, \hspace{0.5cm}\Vert u^n - v \Vert_{C([0, T], H^\sigma(\R^3))} \to 0.
$$
From Sobolev embedding, it follows that for any $T>0$, $u^n|u^n|^2 \to v|v|^2$ in $C([0, T], L^2(\R^3))$, hence $v$ is solution to  (\ref{eq:duha}).
From point (ii) of the claim, it follows that $u = v$ on $[0, T_{\rm max})$. As  $v \in L_{\rm loc}^\infty H^s$, this ends the proof.

\end{proof}

\section{Uniform bound on the energy and local energy decay}
\label{s:nrj}
The main result of this section is the following.
\begin{prop}\label{prpUnifEst}
Assume that the control condition (\ref{eq:control}) holds.
Then, for any $\epsilon>0$, for any $u_0 \in H^{1+\epsilon}(\mathbb R^3)$, the unique maximal solution $u$ to \eqref{eq:DNLS} with data $u_0$ is defined globally forward in time, and verifies, for any $R>0$
$$
\sup_{t>0} E[u(t)] \lesssim \Vert u_0 \Vert_{H^1(\mathbb R^3)}^2 + \Vert u_0 \Vert_{L^4(\R^3)}^4 < \infty, \hspace{0.5cm} \Vert u \Vert_{L^2(\mathbb R_+, L^2(B(0,R)))} \lesssim_R \Vert u_0 \Vert_{H^1(\mathbb R^3)}^2 + \Vert u_0 \Vert_{L^4(\R^3)}^4 < \infty.
$$
\end{prop}

As preliminaries, the mass and energy laws are stated in \S\ref{ss:laws}, and a perturbative (with respect to $G = \operatorname{I})$ Morawetz identity is given in \S\ref{ss:mor}. Proposition \ref{prpUnifEst} is then showed in \S\ref{ss:proof_unif}. Local energy decay statements are given as Corollaries in \S\ref{ss:loc_dec}.
\subsection{Mass and energy laws} \label{ss:laws}

\begin{lem}\label{lm:GWPDef}
        Let $\epsilon >0$, $u_0 \in H^{1 + \epsilon}(\R^3)$, and $u\in C([0,T_{\rm max}),H^{1+\epsilon}(\R^3))$ the unique maximal solution to \eqref{eq:DNLS}. Then for any $0 \leq t_0 \leq t < T_{\rm max}$,
            \begin{equation}\label{eq:MassEvol}
        \| u(t) \|_{L^2(\R^3)}^2 - \| u(t_0)\|_{L^2(\R^3)}^2  = - 2 \int_{t_0}^t \int a(x) |u(s,x)|^2 \,dx \, ds,
    \end{equation}
    in particular,  
    \begin{equation}\label{eq:LocMasDcy}
        \int_0^{T_{\rm max}} \int a(x) |u(s,x)|^2 \, dx \,  ds \leq \frac{ \| u_0 \|_{L^2(\R^3)}^2}{2}.
    \end{equation}
    In addition,
    \begin{align}\label{eq:EnerEvol1}
        E[u(t)] - E[u(t_0)]  &= - \int_{t_0}^t \int a \big( |u|^4 + G \nabla u \cdot \nabla \bar u \big) \, dx  ds + \operatorname{Re}\int  \int_{t_0}^t  G \nabla u \cdot \bar u \nabla a \, dx  ds \\
        &= - \int_{t_0}^t \int a \big( |u|^4 + G \nabla u \cdot \nabla \bar u \big) \, dx  ds - \frac 12 \int  \int_{t_0}^t  |u|^2 \Delta_G a \, dx  ds. \label{eq:EnerEvol2}
    \end{align}
    
\end{lem}
\begin{proof}
We begin with \eqref{eq:EnerEvol1}--\eqref{eq:EnerEvol2}.
    By taking formally the scalar product of \eqref{eq:DNLS} with $\partial_t u$, we obtain (denoting $(f, g) := \operatorname{Re}\int f \bar g\, dx$):
    \begin{align*}
        0 = (\partial_t u, i \partial_t u) = ( \partial_t u, - \Delta_G u + |u|^2 u - i a u) =  \frac{d}{dt} E[u(t)] - (\partial_t u, i a u).
    \end{align*}
    For the last term on the right-hand side, we observe that 
    \begin{align*}
        - (\partial_t u, i a u) = - (i \Delta_G  u - i|u|^2 u - au, i a u) = (|u|^2u, au) + (G\nabla u , a \nabla u) + (G\nabla u, u \nabla a ),
    \end{align*}
    thus we obtain 
    \eqref{eq:EnerEvol1} by time integration and \eqref{eq:EnerEvol2} follows by integration by parts. A classical approximation argument can be used to justify the computations. Similarly, 
    by taking the scalar product of \eqref{eq:DNLS} with $iu$, we get 
    \begin{equation*}
       (i\partial_t u, i u) = \frac{1}{2} \frac{d}{dt} \| u \|_{L^2(\R^3)}^2 = (i \nabla \cdot (G \nabla u ), u) - (au,u) = - (au,u) - (i G\nabla u, \nabla u).
    \end{equation*}
    The last term on the right-hand side is zero because $G$ is symmetric, and we  obtain 
    \eqref{eq:MassEvol} after integration in time.
\end{proof}

\begin{remark}\label{prpBackWP}
    It is possible to obtain  an exponential bound on the $H^1$-norm of solutions (including negative times) in the following way. On the one hand, from Gronwall's inequality and \eqref{eq:MassEvol} we have 
    \begin{equation*}
        M[u(t)] \leq M[u_0] e^{2\| a \|_{L^{\infty}} |t|}.
    \end{equation*}
    On the other hand, \eqref{eq:EnerEvol2} implies, formally, that
$$
        \frac{d}{dt} E[u(t)]  
          \lesssim \| a \|_{C^2(\R^3)} \left( E[u(t)] + M[u(t)] \right).
$$   
    By Gronwall's inequality, we thus obtain 
     \begin{equation*}
        \| \nabla u(t) \|_{L^2}^2 \leq 2 E[u(t)] \lesssim e^{C|t|} E[u_0].
    \end{equation*}
    In what follows, we will show an uniform-in-time upper bound on the energy for positive times.
\end{remark}

\subsection{A perturbative Morawetz identity}  \label{ss:mor}

\begin{prop} \label{prop:Mor}
Let $\chi \in C^4(\mathbb R^3)$ be so that, for any multi-indice $\alpha$ of size $|\alpha| \leq 4$
$$
|\partial^{\alpha} \chi(x)|\lesssim \langle x \rangle^{1- |\alpha| }.
$$
Then, $C>0$ exists so that the following holds.
    Let $\epsilon >0$, $u_0 \in H^{1 + \epsilon}(\R^3)$, and $u\in C([0,T_{\rm max}),H^{1+\epsilon}(\R^3))$ be solution to \eqref{eq:DNLS}. 
    Then for any $0 \leq s \leq t < T_{\rm max}$,
    \begin{equation*}\label{eqVirVir}
    \begin{aligned}
    \operatorname{Im}\int  \bar u(t) \nabla  u(t) \cdot \nabla \chi dx &=  \operatorname{Im}\int  \bar u(s) \nabla  u(s) \cdot \nabla \chi dx + \int_s^t \int  2 D^2 \chi \nabla u \cdot \nabla \bar u - \frac 12 \Delta^2 \chi |u|^2\\ &+ \frac 12 \Delta \chi |u|^4 
     - 2 \operatorname{Im} ( a \bar u \nabla u \cdot \nabla \chi) \, dx d\tau + \int_s^t \mathcal E(\tau)  \, d\tau,
      \end{aligned}
    \end{equation*}
    where
    $$
    | \mathcal E(t) | \leq C \int_{\operatorname{supp}(G - \operatorname{I})} |\nabla u(t, x)|^2 + |u(t, x)|^2 + |u(t, x)|^4\, dx.
    $$
\end{prop}
\begin{proof}
We observe that 
\begin{equation*}
    \frac{d}{dt} \int \operatorname{Im}(\bar{u} \nabla u \cdot \nabla \chi) dx = - \int \operatorname{Im} \left( \bar u \partial_t u \right) \Delta \chi dx - 2 \int \operatorname{Im} \left( \partial_t u \nabla \bar u \right) \cdot \nabla \chi dx =: {\rm I} + {\rm II}.
\end{equation*}
Using equation \eqref{eq:DNLS} we get 
\begin{equation*}
\begin{aligned}
      {\rm I} &= \int (|u|^4 + |\nabla u |^2)\Delta \chi -\frac{1}{2} |u|^2 \Delta^2 \chi  dx + \operatorname{Re} \int (G - \I) \nabla u \cdot \nabla(\bar u \Delta \chi) dx,
\end{aligned}
\end{equation*}
and 
\begin{equation*}
    \begin{aligned}
         {\rm II} &= \int 2  D^2 \chi \nabla u  \cdot \nabla \bar u - \left(\frac{1}{2} |u|^4 + |\nabla u|^2 \right) \Delta \chi + 2a \operatorname{Im}(u \nabla \bar u) \cdot \nabla \chi dx  \\
         & +   \int 2\operatorname{Re} ((G - \I) \nabla u \cdot D^2 \chi  \nabla \bar u )- \operatorname{Re}((G - \I) \nabla u \cdot \nabla \bar u \Delta \chi) - \operatorname{Re} \sum_{1 \leq i,j,k \leq 3} \partial_k G_{ij} \partial_j u \partial_i \bar u \partial_k \chi dx.
    \end{aligned}
\end{equation*}
We get the result by summing the two equations above and observing that for any $k, i,j\in\{ 1,2,3\}$, $ \operatorname{supp} \partial_k G_{i,j} \subset \operatorname{supp} (G -\I) $.
\end{proof}
\subsection{Proof of Proposition \ref{prpUnifEst}}  \label{ss:proof_unif}

\begin{proof}[Proof of Proposition \ref{prpUnifEst}] 
    We will show that there exists a constant $C > 0$ so that, for any $T>0$, if $u$ is solution to \eqref{eq:DNLS} with data $u_0$  defined on a time interval $[0, T]$, then
   $$
\sup_{t\in[0, T]} E[u(t)] + \Vert u \Vert^2_{L^2([0, T], L^2(B(0,R)))} \leq C (\Vert u_0 \Vert^2_{H^1(\mathbb R^3)} + \Vert u_0 \Vert_{L^4(\R^3)}^4).
$$
The conditional global well-posedness theory given as Lemma \ref{prop:cond_glob} then ends the proof.

    Observe that, from Lemma \ref{lm:GWPDef}, \eqref{eq:EnerEvol2}, using the fact that $G$ is positive
    $$
    E[u(t)] \leq E[u(0)] - \frac 12 \int_0^t \int |u|^2 \Delta_G a \, dx ds.
    $$
    Let now $\chi(x) := \sqrt{x^2 + 1}$, and define
    \begin{equation} \label{eq:def_lambda}
    \lambda(t) := \int_0^t \int (-\Delta^2 \chi) |u|^2 dx ds.
    \end{equation}
    As $-\Delta^2 \chi  = \frac{15}{\chi^7}> 0$ and $a$ is compactly supported, $|\Delta_G a| \lesssim -\Delta^2 \chi$ and thefore 
    \begin{equation} \label{eq:bdd_E_Lambda}
    E[u(t)] \leq E[u(0)] + C_0 \lambda(t),
    \end{equation}
    where $C_0>0$ is some constant depending on $a$ and $G$.
    We will show that 
    \begin{equation} \label{eq:Ubound_goal}
     \sup_{t>0}\lambda(t) \lesssim \Vert u_0 \Vert^2_{H^1(\mathbb R^3)} + \Vert u_0 \Vert_{L^4(\R^3)}^4 < \infty,   
    \end{equation}
    from which the result follows thanks to \eqref{eq:bdd_E_Lambda}.
    In order to do so, we let
    $$
    V(t) := \operatorname{Im}\int  u  \nabla \bar u \cdot \nabla \chi dx.
    $$
    By Cauchy-Schwartz inequality, the uniform bound on the mass and \eqref{eq:bdd_E_Lambda}
    \begin{align} \label{eq:Ubound_1}
    |V(t)| \lesssim& M[u(t)]^{\frac 12} E[u(t)]^{\frac 12} \nonumber \\
    \lesssim&  M[u(0)]^{\frac 12} \big(E[u(0)] + C_0 \lambda(t) \big)^{\frac 12}.
    \end{align}
    The plan is to show that, on the other hand, $V(t) \gtrsim \lambda(t) - C(M[u(0)]+E[u(0)])$, from which the result will follow.
    
    By Proposition \ref{prop:Mor},
    \begin{align}
    V(t) &\geq \frac 12 \int_0^t \int D^2 \chi \nabla u \cdot \nabla \bar u dx ds + \lambda(t) \nonumber - C_1
    \int_0^t \int_{\operatorname{supp}(G-\operatorname{I})}  |\nabla u|^2 + |u|^2 + |u|^4 dx ds\nonumber \\
    & - 2\int_0^t \int |a| |\nabla \chi | |u| |\nabla u| dx ds
    - C_2 M[u(0)]^{\frac 12} E[u(0)]^{\frac 12}, \label{eq:Ubound_2}
    \end{align}
    with $C_1, C_2 > 0$ constants depending only on $G$ (and our choice of $\chi$). 
    We first bound the $\int_{\operatorname{supp}(G-\operatorname{I})}$ error term. From Lemma \ref{lm:GWPDef}, \eqref{eq:EnerEvol1},
    \begin{align}
     \int_{0}^t  \int a (|u|^4 +(G\nabla u \cdot \nabla \bar u)) dx ds
         &= \operatorname{Re}  \int_{0}^t \int  (G \nabla u \cdot \bar u \nabla a) dx ds
        + E[u(0)] - E[u(t)] \nonumber \\
        &\leq   \int_{0}^t \int  |G| |\nabla u| |u| |\nabla a| dx ds
        + E[u(0)] \label{eq:Ubound_3}.
    \end{align}
    
    Observe that (see for example \cite[Lemma 4.4]{YaNiChe21}, \cite[Lemma 4.1]{BaMa23}),
    \begin{equation} \label{eq:Cepsi}
    \forall \epsilon >0, \; \exists C_\epsilon>0 \hspace{0.3cm}\text{ s.t. }\hspace{0.3cm}\forall x \in \mathbb R^3, \hspace{0.3cm}|\nabla a(x)| \leq C_\epsilon a(x) + \epsilon.
    \end{equation}
    Indeed, it suffices to show the above for $x \in \operatorname{supp}\nabla a$. By contradiction, if the claim fails, there exists $\epsilon_0>0$ and a sequence $x_n \in  \operatorname{supp}\nabla a$ so that $|\nabla a(x_n)| \geq n a(x_n) + \epsilon_0$. As $\operatorname{supp}\nabla a$ is compact, there exists a subsequence of $x_n$ converging to some $x_\infty \in \mathbb R^3$. As $a \geq 0$ and $\nabla a$ is bounded, we have necessarily $a(x_\infty) = 0$. In particular, $a$ attains a local minimum in $x_\infty$, hence $\nabla a(x_\infty) = 0$, which contradicts $|\nabla a(x_\infty)| \geq \epsilon_0 > 0$.

    Let $\epsilon >0$ to be fixed later, and $C_\epsilon>0$ be given by (\ref{eq:Cepsi}).
    Let now $\psi \in C^\infty_c(\mathbb R^3)$ be so that $\psi =1 $ on $\operatorname{supp}a$, $0\leq \psi \leq 1$, and $\psi$ is supported in
    $\operatorname{supp}a + B(0,1)$.
    By multiplying (\ref{eq:Cepsi}) by $\psi$, we get
    $$
    |\nabla a| \leq C_\epsilon a + \epsilon \psi.
    $$
    Now, from \eqref{eq:Ubound_3} and the above
\begin{align}
\int_{0}^t  \int a (|u|^4 +(G\nabla u \cdot \nabla \bar u)) dx ds  &\leq  {C_\epsilon}  \int_{0}^t \int |G| |\nabla u| |u| |a| dx ds + \epsilon \int_{0}^t \int |G| |\nabla u| |u| \psi dx ds + E[u(0)] \nonumber\\
&\leq {C_\epsilon}  \int_{0}^t \int |G| |\nabla u| |u| |a| dx ds  + \frac \epsilon 2 \int_{0}^t \int |G| |u|^2 \psi dx ds \\
 &\hspace{1cm}  + \frac \epsilon 2 \int_{0}^t \int |G| |\nabla u|^2 \psi dx ds  \nonumber + E[u(0)] \nonumber \\
& \leq {C_\epsilon}  \int_{0}^t \int |G| |\nabla u| |u| |a| dx ds + \epsilon \sup|G| C_3 \lambda(t)  \nonumber \\
 &\hspace{1cm} + \frac \epsilon 2 \int_{0}^t \int |G| |\nabla u|^2 \psi dx ds  + E[u(0)], \label{eq:Ubound_4}
\end{align}
where $C_3>0$ is a constant depending only on $\chi$ and $\psi$ (and hence $\operatorname{supp} a$). Recall also that, from the mass law (Lemma \ref{lm:GWPDef}, (\ref{eq:LocMasDcy}))
\begin{equation}\label{eq:Ubound_4b}
    \int_0^t \int a|u|^2 dx ds  \leq \frac 12 M[u(0)].
\end{equation}
Now, let $\delta_0 >0$ be so that
$$
\forall x \in \operatorname{supp}(G -  \operatorname{I}), \quad a(x) \geq \delta_0.
$$
Observe that
$$
\int_{\operatorname{supp}(G-\operatorname{I})}  |\nabla u|^2 + |u|^2 + |u|^4 dx \leq \frac{1}{\delta_0}  \int a(|\nabla u|^2 + |u|^2 + |u|^4) dx,
$$
and recall that there exists $c_{\rm coerc}>0$ so that $G \xi \cdot \xi \geq c_{\rm coerc} |\xi|^2$, hence denoting $\delta := \min(1, c_{\rm coerc}^{-1}) \delta_0$,
$$
\int_{\operatorname{supp}(G-\operatorname{I})}  |\nabla u|^2 + |u|^2 + |u|^4 dx \leq \frac{1}{\delta}  \int a((G\nabla u \cdot \nabla \bar u) + |u|^2 + |u|^4) dx.
$$
We therefore get, combining the above with \eqref{eq:Ubound_4}, \eqref{eq:Ubound_4b} and \eqref{eq:Ubound_2}
\begin{equation} \label{eq:Ubound_5}
    V(t) \geq {\rm I} + {\rm II},
\end{equation}
where, denoting
$$
A :=  C_1  \sup |G| \delta^{-1},
$$
we have
\begin{equation} \label{eq:Ubound_6}
{\rm I} = \lambda(t) - \epsilon A  C_3 \lambda(t),
\end{equation}
and
\begin{align} \label{eq:Ubound_7}
{\rm II} =&  \frac 12 \int_0^t \int D^2 \chi \nabla u \cdot \nabla \bar u dx ds  - \big(2 + {C_\epsilon A }\big)  \int_{0}^t \int  |\nabla u| |u| |a| dx ds - \frac \epsilon 2 A \int_{0}^t \int  |\nabla u|^2 \psi dx ds  \\
&- E[u(0)]  - C_2 M[u(0)]^{\frac 12} E[u(0)]^{\frac 12} - \frac 12 \delta^{-1} M[u(0)] \nonumber.
\end{align}

Observe that, as $D^2 \chi(x) = \frac{\operatorname{I}}{\chi(x)} - \frac{(x_i x_j)_{i,j}}{\chi(x)^3} \gtrsim \frac{\operatorname{I}}{\chi(x)^3}$ (in the sense of quadratic forms) and $\psi$ is compactly supported we have
\begin{equation} \label{eq:Ubound_8}
\int_{0}^t \int D^2 \chi \nabla u \cdot \nabla \bar u dx ds  \geq \eta \int_{0}^t \int |\nabla u|^2 \psi dx ds ,
\end{equation}
where $\eta = \eta(\chi, \psi)>0$ depends only on $\psi$ and $\chi$.
Let us now fix $\epsilon >0$ small enough so that
$$
\epsilon A   C_3 \leq \frac 12 \quad \text{and} \quad \epsilon A   \leq \frac 12 \eta.
$$
We get from \eqref{eq:Ubound_5}, \eqref{eq:Ubound_6}, \eqref{eq:Ubound_7} together with \eqref{eq:Ubound_8}
\begin{equation} \label{eq:Ubound_9}
{\rm I} \geq \frac 12 \lambda(t)
\end{equation}
and
\begin{align}  \label{eq:Ubound_10}
{\rm II} \geq& \frac 14 \eta \int_{0}^t \int |\nabla u|^2 \psi dx ds   - \big(2 + {C_\epsilon A}\big)  \int_{0}^t \int  |\nabla u| |u| |a| dx ds  \\
 & - E[u(0)]  - C_2 M[u(0)]^{\frac 12} E[u(0)]^{\frac 12} - \frac 12 \delta^{-1} M[u(0)] \nonumber.
\end{align}

We will now show that
\begin{equation}  \label{eq:Ubound_11}
{\rm II} \geq - C \times \big( M[u(0)] + E[u(0)]\big),
\end{equation}
for a universal constant $C>0$.
Indeed, by Cauchy-Schwartz inequality, then the mass law (Lemma \ref{lm:GWPDef}, \ref{eq:LocMasDcy})
\begin{align*}
  \int_{0}^t \int |\nabla u| |u| |a| dx ds &\leq \Big( \int_{0}^t \int a |\nabla u |^2 dx ds \Big)^{\frac 12} \Big( \int_{0}^t \int a |u|^2 dx ds \Big)^{\frac 12} \\ 
  &\leq \frac 12 M[u(0)]^{\frac 12} \Big( \int_{0}^t \int a |\nabla u |^2 dx ds  \Big)^{\frac 12},  
\end{align*}
hence, as $\psi = 1$ on the support of $a$,
\begin{equation}  \label{eq:Ubound_12}
\int_{0}^t \int |\nabla u| |u| |a|dx ds  \leq \frac 12 M[u(0)]^{\frac 12} \sup |a| \Big( \int_{0}^t \int  |\nabla u |^2 \psi dx ds  \Big)^{\frac 12}.
\end{equation}

Plugging the above in \eqref{eq:Ubound_10}, we obtain
$$
{\rm II} \geq F \Bigg(\Big(\int_{0}^t \int |\nabla u|^2 \psi \Big)^{\frac 12}\Bigg),
$$
where
\begin{align*}
  F(X) :=& \frac 14 \eta X^2 - \big(1 + \frac{C_\epsilon A}{2}\big)  M[u(0)]^{\frac 12} \sup |a| X \\
 &- E[u(0)]  
 - C_2 M[u(0)]^{\frac 12} E[u(0)]^{\frac 12} - \frac 12 \delta^{-1} M[u(0)].
\end{align*}
The fact that the parabola of equation $Y = F(X)$ is always above its vertex gives 
\eqref{eq:Ubound_11} (with $C>0$ depending on $A$, $\epsilon$, $\eta$, and $\sup |a|$).

Finally, combining \eqref{eq:Ubound_5} with \eqref{eq:Ubound_9}, \eqref{eq:Ubound_11} and \eqref{eq:Ubound_1}, we obtain
$$
\lambda(t)^2 - b \times \big( M[u(0)] + E[u(0)]\big) \lambda(t) - c \times \big( M[u(0)] + E[u(0)]\big)^2 \leq 0,
$$
with $b, c>0$, universal constants. The estimate (\ref{eq:Ubound_goal}), and hence the result, follows.
\end{proof}

\subsection{Local energy decay}  \label{ss:loc_dec}

\begin{cor} \label{cor:Ubound1}
    Assume that (\ref{eq:control}) holds.
Then, for any $R>0$ there is $C>0$ so that for any solution $u$ to \eqref{eq:DNLS}, 
\begin{equation}\label{eq:L2H1LocDec}
    \Vert u \Vert^2_{L^2(\mathbb R_+, H^1(B(0,R)))} \leq C(\Vert u_0 \Vert_{H^1(\mathbb R^3)}^2 + \Vert u_0 \Vert_{L^4(\R^3)}^4) < \infty.
\end{equation}
\end{cor}
\begin{proof}
Let $T>0$. Observe that, by Lemma \ref{lm:GWPDef}, \eqref{eq:EnerEvol2} 
$$
    \int_0^T \int a(|\nabla u|^2 +  |u|^{4})  \, dx \,ds = E[u_0] - E[u(T)] - \frac{1}{2}   \int_0^T \int (\Delta_G a) |u|^2 \, dx \,ds,
$$   
hence Proposition \ref{prpUnifEst} implies that
$$
    \int_0^\infty \int a(|\nabla u|^2 +  |u|^{4})  \, dx \,ds  
\lesssim \Vert u_0 \Vert_{H^1(\mathbb R^3)}^2 + \Vert u_0 \Vert_{L^4(\R^3)}^4 < +\infty.
$$    
Now, with $\chi \in C^\infty$  defined as in the proof of Proposition \ref{prpUnifEst}, observe that  (\ref{eq:Ubound_2}) together with  (\ref{eq:Ubound_goal}), the above and the fact that $a(x) \geq \delta_0 > 0$ in $\operatorname{supp}(G-\operatorname{I})$ yields
$$
    \int_0^\infty \int D^2 \chi \nabla u \cdot \nabla u  \, dx \,ds  
\lesssim \Vert u_0 \Vert_{H^1(\mathbb R^3)}^2 + \Vert u_0 \Vert_{L^4(\R^3)}^4.
$$
The result follows as $D^2 \chi \geq 0$ and $D^2 \chi \geq C_R$ in $B(0,R)$.
\end{proof}

\begin{cor} \label{cor:Ubound2}
    Assume that (\ref{eq:control}) holds.
Then, for any solution u to \eqref{eq:DNLS} and any $R>0$
$$
\Vert u(t) \Vert_{H^{s}(B(0,R))} \rightarrow 0.
$$
for any $s \in [0,1)$
\end{cor}
\begin{proof}
Let $\chi \in C^\infty_c(\R^3)$ be arbitrary.
    Observe that, by Corollary \ref{cor:Ubound1}, there exists $t_n \to +\infty$ so that
    $$
    \Vert \chi u(t_n) \Vert_{H^1(\mathbb R^3)} \to 0.
    $$
    But
    $$
    \partial_t \Vert \chi u \Vert_{L^2(\R^3)}^2 = - \int \chi^2 a |u|^2 dx - \operatorname{Im} \int  G\nabla \chi^2 \cdot \nabla u \bar u dx,
    $$
    from which, for any $t\geq t_n$
    $$
     \Vert \chi u(t) \Vert_{L^2(\R^3)}^2 \leq \Vert \chi u(t_n) \Vert_{L^2(\R^3)}^2 + \int_{t_n}^t \int |G\nabla \chi^2 \cdot \nabla u \bar u|dx ds .
    $$
    In particular, taking $\widetilde \chi \in C^\infty_c(\R^3)$ so that $\widetilde \chi = 1$ on $\operatorname{supp}\chi$
        $$
     \Vert \chi u(t) \Vert_{L^2(\R^3)}^2 \leq \Vert \chi u(t_n) \Vert_{L^2(\R^3)}^2 + C \int_{t_n}^{\infty} \int \widetilde \chi (|\nabla u|^2 + |u|^2)dx ds .
    $$
    The above goes to zero as $n \to \infty$ thanks to Corollary \ref{cor:Ubound1}, therefore
    $$
    \Vert \chi u(t) \Vert_{L^2(\R^3)} \to 0.
    $$
    as $t \to \infty$. The result follows by interpolation: indeed, by Plancherel, for $t\geq 0$, using the fact that $u$ is bounded in $H^1(\mathbb R^3)$ on $\mathbb R_+$
    $$
    \Vert D^{\frac 12} (\chi u) \Vert^2_{L^2(\R^3)} \leq \Vert \nabla(\chi u)\Vert_{L^2(\R^3)} \Vert \chi u\Vert_{L^2(\R^3)} \lesssim \Vert \widetilde \chi u \Vert_{L^2(\R^3)} \to 0,
    $$
thus
    $$
        \Vert D^{ s} (\chi u) \Vert^2_{L^2(\R^3)} \leq \Vert \nabla( \chi  u) \Vert^{2s}_{L^2(\R^3)} \Vert \chi u\Vert^{2(1-s)}_{L^2(\R^3)} ,
    $$
    so  we get $\Vert D^{s} (\chi u) \Vert_{L^2(\R^3)} \to 0$ for any $0\leq s<1$.
    \end{proof}

    \begin{remark}[The linear case] \label{rk:lin_dec}
        Strictly analogous computations as the ones presented in this Section can be carried out in the linear case. This shows local energy decay for the linear damped flow $S(t)$, in the form that $\Vert S(t) u_0 \Vert_{L^2(\mathbb R_+, H^1(\mathbb R^3))} \lesssim \Vert u_0 \Vert_{H^1}$, and hence provide an alternative proof of the result of \cite{AlKh07} in our setting, under the (strongest) control condition (\ref{eq:control}).
    \end{remark}

\section{Bilinear estimates} \label{s:bil}
The purpose of this section is to show the following 
\begin{prop} \label{prop:L4L4}
    Assume that (\ref{eq:control}) holds. Then, for any $\epsilon>0$, any solution $u\in C([0,\infty), H^{1+\epsilon}(\R^3))$ to \eqref{eq:DNLS} verifies 
    \begin{equation}\label{eq:GlStrEst}
        \int_0^\infty \int |u(x,t)|^4 dx dt \lesssim 1.
    \end{equation}
\end{prop}
\begin{proof}
The overall idea is to perform a bilinear virial computation similar to the one known in $\mathbb R^3$ (see for example \cite[Proposition $2.5$]{CoKeStTaTa04}), perturbatively with respect to $G= \operatorname{I}$, and to handle the error terms, localized where $G \neq \operatorname{I}$, thanks to Corollary \ref{cor:Ubound1}. To this purpose, for $x \in \mathbb R^3$, define $\rho(x) := |x|$,
and
\begin{equation*}
    B(t) := \int |u(y)|^2 \int \operatorname{Im}(\bar u(x) \nabla u(x)) \cdot \nabla \rho(x-y)dx dy.
\end{equation*}
Observe for later use that
$$
\nabla \rho = \frac{x}{|x|}, \hspace{0.3cm} \Delta \rho = \frac{2}{|x|}, \hspace{0.3cm} \nabla \Delta \rho = \frac{2x}{|x|^3}, \hspace{0.3cm} \Delta^2 \rho = -8\pi \delta_{x=0}.
$$

We have
\begin{equation}\label{eq:Bil3Terms}
    \begin{aligned}
        \frac{d}{dt} B(t) &= 2 \iint  \operatorname{Re}(\bar u(y) \partial_t u(y)) \operatorname{Im}(\bar u(x) \nabla u(x)) \cdot \nabla \rho(x-y) dx dy \\ 
        &+ \iint |u(y)|^2 \operatorname{Im}(\partial_t \bar u(x) \nabla u (x)) \cdot \nabla \rho(x-y) dx dy \\
        &+ \iint |u(y)|^2 \operatorname{Im}( \bar u(x) \nabla \partial_t u (x)) \cdot \nabla \rho(x-y) dx dy \\
        & = 2 \iint \operatorname{Re}(\bar u(y) \partial_t u(y)) \operatorname{Im}(\bar u(x) \nabla u(x)) \cdot \nabla \rho(x-y) dx dy \\
        & - \iint |u(y)|^2 \operatorname{Im}(\bar u(x) \partial_t u(x)) \Delta \rho(x-y) dx dy \\
        &- 2 \iint |u(y) |^2 \operatorname{Im}(\partial_t u(x) \nabla \bar u(x)) \cdot \nabla \rho(x-y) dx dy \\
        &=: {\rm I} + {\rm II} + {\rm III}.
    \end{aligned}
\end{equation}
Using equation \eqref{eq:DNLS}, we compute the three quantities above separately. We have 
\begin{equation*}
    \begin{aligned}
        \frac{1}{2} {\rm I} &=  \iint  \operatorname{Re}(\bar u(y) (i\Delta_G u(y) - i |u(y)|^{4} u(y) - a(y) u(y))) \operatorname{Im}(\bar u(x) \nabla u(x)) \cdot \nabla \rho(x-y) dx dy \\
        &= - \iint a(y) |u(y)|^2 \operatorname{Im}(\bar u(x) \nabla u(x)) \cdot \nabla \rho(x-y) dx dy \\
        & + \iint \operatorname{Im}(\bar u(x) \nabla u(x)) \cdot \nabla \rho(x-y) \operatorname{Im}(\nabla \bar u(y) \cdot (G(y)\nabla u(y)) dx dy \\
        & + \iint \operatorname{Im} ( \bar u(y) (G(y) \nabla u(y)) \cdot \nabla_y \big(\operatorname{Im}(\bar u (x) \nabla u(x)) \cdot \nabla \rho(x-y)\big) dx dy \\
        &=: {\rm I a}+ {\rm Ib} + {\rm Ic}.
    \end{aligned}
\end{equation*}
For the first term, we observe that
\begin{equation*}
    \begin{aligned}
        |{\rm Ia}| &\leq \iint a(y) |u(x)| |\nabla u(x)| \nabla \rho (x - y) |u(y)|^2 dx dy \\
        & \lesssim  \| u \|_{H^\frac{1}{2}}^2 \int a(y) |u(y)|^2 dy
    \end{aligned}
\end{equation*}
as long as $|\nabla \rho (x-y)| \leq C < \infty$. The second term ${\rm Ib}$ equals zero since $G$ is symmetric. 
For the last term, we split the integral as 
\begin{equation*}
    \begin{aligned}
        {\rm Ic} &= \iint \operatorname{Im} ( \bar u(y) ((G(y) - \I)\nabla u(y)) \cdot \nabla_y \big(\operatorname{Im}(\bar u (x) \nabla u(x)) \cdot \nabla \rho(x-y)\big) dx dy \\
        &+ \iint\operatorname{Im} ( \bar u(y) \nabla u(y)) \cdot \nabla_y \big(\operatorname{Im}(\bar u (x) \nabla u(x)) \cdot \nabla \rho(x-y)\big) dx dy \\
        &=: {\rm Ic1} + {\rm Ic2}.
    \end{aligned}
\end{equation*}
Here, in the first term, $y$ is localized in space, and we compute by the Cauchy-Schwartz inequality
\begin{equation*}
    \begin{aligned}
        |{\rm Ic1}| &\lesssim \iint_{y \in \operatorname{supp}(G - \operatorname{I})}  |D^2 \rho(x - y)| | u(x)| |\nabla u(x)| |u(y) |\nabla u(y)|   dx dy \\ 
        &\leq \int_{y \in \operatorname{supp}(G - \operatorname{I})}  \| D^2 \rho(. - y)  u \|_{L^2} \|  \nabla u \|_{L^2} |u(y) |\nabla u(y)| dy \\ 
        & \lesssim  \|  \nabla u \|_{L^2}^2 \int_{y \in \operatorname{supp}(G - \operatorname{I})}   |u(y) |\nabla u(y)| dy \\
        & \lesssim \|  \nabla u \|_{L^2}^2 \int_{y \in \operatorname{supp}(G - \operatorname{I})} |\nabla u(y)|^2 + | u(y)|^2 dy,
    \end{aligned}
\end{equation*}
where we used that, thanks to  the Hardy inequality, $\| D^2 \rho(. - y)  u \|_{L^2} \lesssim \| \frac{1}{|x|}  u \|_{L^2} \lesssim \|  \nabla u \|_{L^2}$. For the term ${\rm Ic2}$, 

from our choice of $\rho$, we obtain, similarly to \cite[Proposition $2.5$]{CoKeStTaTa04},  
\begin{equation*}
\begin{aligned}
        {\rm Ic2} & \geq -  \iint |u(x)| |\nabla u(x)| |u(y)| |\nabla u(y)| \frac{dxdy}{|x-y|} \\
         & \geq - \frac 12 \iint | u(y)|^2 |\nabla u(x) |^2 \frac{dxdy}{|x-y|}.
\end{aligned}
\end{equation*}
For the second term in \eqref{eq:Bil3Terms}, we use equation \eqref{eq:DNLS} to obtain

\begin{equation*}
    \begin{aligned}
        {\rm II} &= \iint |u(y)|^2 |u(x)|^{4} \Delta \rho(x-y) dx dy + \iint |u(y)|^2 \operatorname{Re}( \nabla \bar u(x) \cdot (G(x) \nabla u(x))) \Delta \rho(x-y) dx dy \\ 
        & + \iint |u(y)|^2 \operatorname{Re}(\bar u (x) (G(x) \nabla u(x) ) \cdot \nabla_x \Delta \rho(x - y) dx dy  \\
        & =:  {\rm IIa} + {\rm IIb} + {\rm IIc}.
    \end{aligned}
\end{equation*}
We split
$$
\begin{aligned}
{ \rm IIb } &= \iint |u(y)|^2  |\nabla \bar u(x)|^2 \Delta \rho(x-y) dx dy
+
\iint |u(y)|^2 \nabla \bar u(x) \cdot (G(x) - \I) \nabla u(x) \Delta \rho(x-y) dx dy \\
&=: { \rm IIb1 } + { \rm IIb2 }
\end{aligned}
$$
Now we observe that, with our choice of $\rho$, 
$$
{ \rm IIb1} + 2 {\rm Ic2} \geq \frac{1}{2}{\rm IIb1},
$$
and, from Hardy inequality, similar to the above
$$
|{ \rm IIb2}|  \lesssim \| u \|_{H^1(\mathbb R^3)}^2 \int_{x \in \operatorname{supp}(G-\operatorname{I})} ( |u(x)|^2 + |\nabla u(x)|^2) dx.
$$ 
For the third term in the equation above, we separate the integral in $x$ and get
\begin{equation*}
    \begin{aligned}
        {\rm IIc} &= - \frac{1}{2} \iint |u(y)|^2 |u(x)|^2 \Delta \Delta \rho(x-y) dx dy \\
        & + \iint  |u(y)|^2 \operatorname{Re}(\bar u (x) ((G(x) - \I) \nabla u(x) ) \cdot \nabla_x \Delta \rho(x - y) dx dy \\
        & =: {\rm IIc1} + {\rm IIc2},
    \end{aligned}
\end{equation*}
and for the last term above we have 
\begin{equation*}
    \begin{aligned}
        |{\rm IIc2}| &\leq \int_{x \in \operatorname{supp}(G-\operatorname{I})} |u(x)| |(G(x) - \I) \nabla u(x)|  \int |u(y)|^2 |\nabla \Delta \rho(x - y)| dy dx \\ 
       & \lesssim \| u \|_{H^1(\mathbb R^3)}^2 \int_{x \in \operatorname{supp}(G-\operatorname{I})} ( |u(x)|^2 + |\nabla u(x)|^2) dx,
    \end{aligned}
\end{equation*}
due to the Hardy inequality. 

Now we deal with the third term in \eqref{eq:Bil3Terms}. By using \eqref{eq:DNLS} we obtain 
\begin{equation*}
    \begin{aligned}
        \frac{1}{2} {\rm III} & = \iint |u(y)|^2 a(x) \operatorname{Im}(\nabla \bar u(x) u(x)) \cdot \nabla \rho(x-y) dx dy \\
        & - \frac{1}{4} \iint |u(y)|^2 |u(x)|^{4} \Delta \rho(x-y) dx dy \\ 
        & + \iint |u(y)|^2 \operatorname{Re} ((G(x) \nabla u(x)) \cdot \nabla_x( \nabla \bar u(x) \cdot \nabla \rho(x-y)) dx dy \\
        & =: {\rm IIIa} + {\rm IIIb} + {\rm IIIc}.
    \end{aligned}
\end{equation*} 
For the first term in the above equation, we observe that 
\begin{equation*}
    |{\rm IIIa}| \leq \frac{1}{2}\| u \|_{L^2}^2 \int a(x) (|u(x)|^2 + |\nabla u(x)|^2 ) dx.
\end{equation*}
For the last term,  
\begin{equation*}
    \begin{aligned}
        {\rm IIIc} & = \iint |u(y)|^2 \operatorname{Re}(G(x) \nabla u(x)) \cdot D^2 \rho(x - y)  \nabla u(x) dx dy \\ 
        & + \operatorname{Re} \iint |u(y)|^2 \sum_{i,j,k}^3 G_{i,j}(x) \partial_j u(x) \partial_{i,k} \bar u(x) \partial_k \rho(x-y) dx dy \\ 
        & =: {\rm IIIc1} + {\rm IIIc2}. 
    \end{aligned}
\end{equation*}
From Hardy's inequality, we get
$$
{\rm IIIc1} = \iint |u(y)|^2 \nabla u(x) \cdot D^2 \rho(x - y)  \nabla u(x) dx dy + {\rm IIIc1(ii)},
$$
with
$$
|{\rm IIIc1(ii)}| \lesssim \| u \|_{L^2}^2 \int_{x \in \operatorname{supp}(G - \operatorname{I})} |\nabla u(x)|^2 dx.
$$
For the term ${\rm IIIc2}$, we observe that by integration by parts 
\begin{equation*}
    \begin{aligned}
        {\rm IIIc2} &= - \operatorname{Re} \iint |u(y)|^2 \sum_{i,j,k}^3  \partial_k (G_{i,j}(x))  \partial_j u(x) \partial_{i} \bar u(x) \partial_k \rho(x-y) dx dy \\
        &- \iint |u(y)|^2 \operatorname{Re} ( \nabla \bar u(x) \cdot G(x) \nabla u(x)) \Delta \rho(x - y)  dx dy \\
        & - \iint |u(y)|^2 \sum_{i,j,k}^3  G_{i,j}(x) \partial_{j,k} u(x) \partial_{i} \bar u(x) \partial_k \rho(x-y) dx dy. 
    \end{aligned}
\end{equation*}
The last term on the equation's right-hand side above equals ${\rm IIIc2}$. This becomes evident if one swaps the indexes $i$ and $j$ and later observes that $G_{i,j} = G_{j, i}$ as $G$ is symmetric.  Thus we obtain 
\begin{equation*}
    \begin{aligned}
        {\rm IIIc2} &= -\frac{1}{2} \operatorname{Re} \iint |u(y)|^2 \sum_{i,j,k}^3  \partial_k (G_{i,j}(x))  \partial_j u(x) \partial_{i} \bar u(x) \partial_k \rho(x-y) dx dy \\
        &- \frac{1}{2} \iint |u(y)|^2 \operatorname{Re} ( \nabla \bar u(x) \cdot (G(x)-\I) \nabla u(x)) \Delta \rho(x - y)  dx dy  \\ 
        &\- \frac{1}{2} \iint |u(y)|^2  |\nabla \bar u(x)|^2 \Delta \rho(x - y)  dx dy  \\
        & =: {\rm IIIc2(i)} +  {\rm IIIc2(ii)}+ {\rm IIIc2(iii)}.
    \end{aligned}
\end{equation*}
Now we observe that ${\rm IIIc2(i)}$ is localized in space in the $x$ variable as $\partial_k G_{i,j}(x) \equiv 0 $ when $x \notin \operatorname{supp}(G - \operatorname{I})$. So we get, using Hardy's inequality for $ {\rm IIIc2(ii)}$ we obtain
\begin{equation*}
    |{\rm IIIc2(i)}|+|{\rm IIIc2(ii)}| \lesssim \| u \|_{L^2}^2 \int_{x \in \operatorname{supp}(G - \operatorname{I})} |\nabla u(x)|^2 dx.
\end{equation*}
Putting together every term above, we reach the following conclusion: there exists a universal constant $C >0$ such that 

\begin{equation*}
    \begin{aligned}
        C \| u \|_{H^1(\mathbb R^3)}^2 \int_{\operatorname{supp} a} |\nabla u |^2 + |u|^2 dx + \frac{d}{dt} B(t)  & \geq \frac{1}{2} \iint |u(y)|^2 |\nabla \bar u(x)|^2  \Delta \rho(x-y) dxdy \\
        & + 4\pi \int |u|^4 dx \\
        & + \frac{3}{4} \iint |u(y)^2 |u(x)|^{4} \Delta \rho(x - y) dx dy \\ 
        & + \iint
        |u(y)|^2 D^2\rho (x - y)\nabla \bar u(x) \cdot \nabla u(x) dx dy.
    \end{aligned}
\end{equation*}
From the positivity of $D^2\rho$, every term on the right-hand side above is non-negative. Thus, after integration in time, we obtain 
\begin{equation*}
\begin{aligned}
     4\pi \int_0^T \int |u|^4 dx dt & \leq C \| u \|_{H^1(\mathbb R^3)}^2 \int_0^T \int_{\operatorname{supp} a} |\nabla u |^2 + |u|^2 dx dt + B(T)  - B(0) \\
     & \leq C + \| u (T) \|_{L^2(\R^3)}^2 \| u(T) \|_{\dot{H}^\frac{1}{2}(\R^3)}^2 - \| u (0) \|_{L^2(\R^3)}^2 \| u(0) \|_{\dot{H}^\frac{1}{2}(\R^3)}^2,
\end{aligned}
\end{equation*}
where the last inequality follows from Corollary \ref{cor:Ubound1}. From the uniform boundedness of the $H^1$-norm given by Proposition \ref{prpUnifEst}, we thus obtain \eqref{eq:GlStrEst}.
\end{proof}

    \begin{remark}[Linear Stichartz estimates]
    Similarly to Remark \ref{rk:lin_dec}, the same computation can be performed in the linear case. This shows that the damped linear flow $S(t)$ verifies the global Strichartz-type estimate $\Vert S(t) u_0 \Vert_{L^4(\R_+, L^4(\R^3))} \lesssim \Vert u_0 \Vert_{H^1(\R ^3)}$. As the scale-invariant estimate is $\Vert e^{it\Delta} u_0 \Vert_{L^4(\R_+, L^4(\R^3))} \lesssim \Vert u_0 \Vert_{H^{\frac 14}(\R^3)}$, this corresponds to a loss of $\frac 34 = \frac 14 + \frac 12 (= \frac 1p + \frac 12)$ derivatives, which is consistent with Remark \ref{rk:dram}.
\end{remark}

\section{End of the proof}\label{secEnd}
\begin{proof}[Proof of Theorem \ref{thm:ScatDef}]
We first show scattering in $H^{\frac 12}$, then upgrade it to $H^s$ for any $s<1$ by an interpolation argument.

    Let $\chi \in C^\infty_c(\R^3)$ be so that $\chi = 1$ on $\operatorname{supp} a$. We let
    $$
w := (1-\chi)u.
    $$
    Then, using the fact that $1-\chi = 0$ on $\operatorname{supp} a$ and hence also on $\operatorname{supp} (G-\operatorname{I})$, $w$ solves the equation
    $$
    i\partial_t w - \Delta w = [\Delta, \chi] u - (1-\chi)u|u|^{2},
    $$
    hence from Duhamel's formula
    $$
    w(t) = e^{it \Delta}\Big( (1-\chi) u_0 + i \int_0^t e^{-is\Delta} (1-\chi)u|u|^{2}(s) \, ds - i \int_0^t e^{-is\Delta}[\Delta, \chi] u(s) \, ds \Big).
    $$
    Therefore, as $\Vert \chi u(t) \Vert_{H^{\frac 12}(\R^3)} \to 0$ thanks to Corollary \ref{cor:Ubound2}, in order to conclude, it suffices to show that
    \begin{equation} \label{eq:scat_crit_I1}
            \Big\Vert  \int_t^\infty e^{-is\Delta} (1-\chi)u|u|^{2}(s) \, ds \Big\Vert_{H^{\frac 12}(\mathbb R^3)} \to 0 \hspace{0.3cm}\text{ as }\hspace{0.3cm}t\to +\infty
    \end{equation}
    and
    \begin{equation}\label{eq:scat_crit_I2}
            \Big\Vert  \int_t^\infty e^{-is\Delta}[\Delta, \chi] u(s) ds \Big\Vert_{H^{\frac 12}(\mathbb R^3)} \to 0 \hspace{0.3cm}\text{ as }\hspace{0.3cm}t\to +\infty,
    \end{equation}
    indeed, the result  then follows by setting
    $$
    u_+ :=  (1-\chi) u_0 + i \int_0^\infty e^{-is\Delta} (1-\chi)u|u|^{2}(s) \, ds - i \int_0^\infty e^{-is\Delta}[\Delta, \chi] u(s) \, ds.
    $$
    First, observe that (\ref{eq:scat_crit_I2}) is a consequence of Corollary \ref{cor:Ubound1}, thanks to the dual estimate to the smoothing effect for the free Schrödinger flow in $\mathbb R^3$. More precisely, selecting $\widetilde \chi \in C^\infty_c(\mathbb R^3)$ so that $\widetilde \chi = 1$ on $\operatorname{supp} \chi$ and arguing by duality,
    \begin{align*}
        &\sup_{\substack{F \in L^2(\mathbb R^3), \\ \Vert F \Vert_{L^2(\R^3)} \leq 1}} \big\langle F, \langle D \rangle^{\frac 12} \int_t^\infty e^{-is\Delta}[\Delta, \chi] u(s) ds \big\rangle_{L^2} = \sup_{\substack{F \in L^2(\mathbb R^3), \\ \Vert  F \Vert_{L^2(\R^3)} \leq 1}} \big\langle \langle D \rangle^{\frac 12} F,  \int_t^\infty e^{-is\Delta}[\Delta, \chi] u(s) ds \big\rangle_{L^2} \\
        &\hspace{1cm} = \sup_{\substack{F \in L^2(\mathbb R^3), \\ \Vert  F \Vert_{L^2(\R^3)} \leq 1}} \big\langle \langle D \rangle^{\frac 12} F,  \int_t^\infty e^{-is\Delta}\widetilde \chi [\Delta, \chi] u(s) ds \big\rangle_{L^2} \\
        &\hspace{1cm}= \sup_{\substack{F \in L^2(\mathbb R^3), \\ \Vert  F \Vert_{L^2(\R^3)} \leq 1}} \big\langle \widetilde \chi e^{is\Delta}\langle D \rangle^{\frac 12} F,   [\Delta, \chi] u(s)  \big\rangle_{L_s^2((t, +\infty), L^2(\mathbb R^3))} \\
        &\hspace{1cm}\leq \sup_{\substack{F \in L^2(\mathbb R^3), \\ \Vert  F \Vert_{L^2(\R^3)} \leq 1}} \Vert \widetilde \chi e^{is\Delta}\langle D \rangle^{\frac 12} F \Vert_{L_s^2((t, +\infty), L^2(\mathbb R^3))} \Vert  [\Delta, \chi] u(s)  \Vert _{L_s^2((t, +\infty), L^2(\mathbb R^3))} \\
        &\hspace{1cm}\lesssim \Vert  [\Delta, \chi] u(s)  \Vert_{L_s^2((t, +\infty), L^2(\mathbb R^3))} \lesssim \Vert \widetilde \chi u(s)\Vert_{L_s^2((t, +\infty),H^1(\mathbb R^3))} \to 0,
    \end{align*}
    where we used Kato's smoothing effect in the last line, then Corollary \ref{cor:Ubound1} to conclude.
   
    Finally, to show (\ref{eq:scat_crit_I1}), observe that, thanks to Proposition \ref{prop:L4L4}
     \begin{equation} \label{eq:W11}
     \Vert u|u|^2 \Vert_{L^2(\R_+, W^{1, 1}(\R^3))} \leq \Vert u \Vert_{L^\infty(\R_+, H^1(\mathbb R^3))} \Vert u \Vert^2_{L^4(\R_+, L^4(\R^3))} < +\infty.
    \end{equation}
    But we have the Sobolev embedding $W^{1, 1} (\R^3)\hookrightarrow W^{\frac 12, \frac {6}{5}}(\R^3)$, and $(2, \frac {6}{5})$ is the dual of the (endpoint) Strichartz pair $(2, 6)$ for the free Schrödinger flow in $\mathbb R^3$. Therefore, 
     an application of the (dual) Strichartz estimate in $\mathbb R^3$ gives the result: more precisely, arguing by duality again
         \begin{align*}
        &\sup_{\substack{F \in L^2(\mathbb R^3), \\ \Vert F \Vert_{L^2(\R^3)} \leq 1}} \big\langle F, \langle D \rangle^{\frac 12} \int_t^\infty e^{-is\Delta} (1-\chi)u|u|^{2}(s) \, ds \big\rangle_{L^2} \\
        &\hspace{1cm}= \sup_{\substack{F \in L^2(\mathbb R^3), \\ \Vert  F \Vert_{L^2(\R^3)} \leq 1}} \big\langle e^{is\Delta} F,  \langle D \rangle^{\frac 12}(1-\chi)u|u|^{2} \big\rangle_{L_s^2((t, +\infty), L^2(\mathbb R^3))} \\
        &\hspace{1cm}\leq \sup_{\substack{F \in L^2(\mathbb R^3), \\ \Vert  F \Vert_{L^2(\R^3)} \leq 1}} \Vert  e^{is\Delta} F \Vert_{L^2((t, +\infty), L^6(\R^3))} \Vert  \langle D \rangle^{\frac 12} (1-\chi)u|u|^{2}  \Vert _{L_s^2((t, +\infty), L^{\frac 65}(\mathbb R^3))} \\
        &\hspace{1cm}\lesssim \Vert  (1-\chi)u|u|^{2}  \Vert _{L_s^2((t, +\infty), W^{\frac 12, \frac 65}(\mathbb R^3))} \to 0,
    \end{align*}
    where in the last line we used Strichartz estimates for the free Schr\"odinger flow in $\mathbb R^3$, then (\ref{eq:W11}) together with the embedding $W^{1, 1}(\R^3) \hookrightarrow W^{\frac 12, \frac {6}{5}}(\R^3)$ to conclude. This ends the proof of scattering in $H^{\frac 12}$.
    
    Now, as $t \to +\infty$
    $$
    \Vert u_+ - e^{-it\Delta} u(t) \Vert_{H^{\frac 12}(\R ^3)} \to 0,
    $$
    and $e^{-it\Delta} u(t)$ is uniformly bounded in $H^1$ forward in time, so, by 	uniqueness of the weak limit, 
    $u_+ \in H^1$. By interpolation, it follows that
    $$
    \Vert u_+ - e^{-it\Delta} u(t) \Vert_{H^{s}(\R ^3)} \to 0,    
    $$
    for any $s<1$. This ends the proof.

    \end{proof}

\appendix

\section{The exterior of an obstacle} \label{app:obst}

The purpose of this Appendix is to illustrate that the non-linear arguments presented in  \S\ref{s:nrj}, \S\ref{s:bil}, \S\ref{secEnd}, work just as well for the analogous problem posed in the exterior of a Dirichlet obstacle.

More precisely, let us consider the Cauchy problem
\begin{equation} \label{eq:DNLS_o}
\begin{cases}
    i\partial_t u  + \Delta u + iau = |u|^2 u, \\
    u = 0 \text{ on }\partial \Omega,\\
    u(0) = u_0 \in H^s(\R^3),
\end{cases}
\end{equation}
where $\Omega := \mathbb R^3 \backslash \Theta$, with $\Theta \subset \mathbb R^3$ compact with smooth boundary. The best general (i.e., not depending on the geometry of $\Theta$) Strichartz estimates for the linear Schr\"odinger equation outside $\Theta$ known at the moment involve a loss of $\frac{3}{2p}$ derivatives \cite{An08}, or a loss of $\frac 1p$ derivatives but for a restricted range of admissible couples \cite{BlSmSo12}, and this is not enough to obtain a well-posedness theory such as presented in \S\ref{secPrelim}. On the other hand, the rest of the proof works similarly, and we obtain:
\begin{thm}
    Assume that $\big\{ \partial \Theta \cdot n(x) < 0\big\} \Subset \big\{ a > 0\big\} $. Then, there exists a $C>0$, such that for any $T>0$, any $C([0, T], H^2(\Omega))$ solution to (\ref{eq:DNLS_o})
    {is bounded in $L^\infty H^1$ by $C\Vert u_0 \Vert_{H^1(\Omega)}$},
    and global  $C([0, T], H^2)$ solutions scatter in $H^{1-}(\R^3)$.
\end{thm}
\begin{figure}
\begin{center}
\includegraphics[scale=0.75]{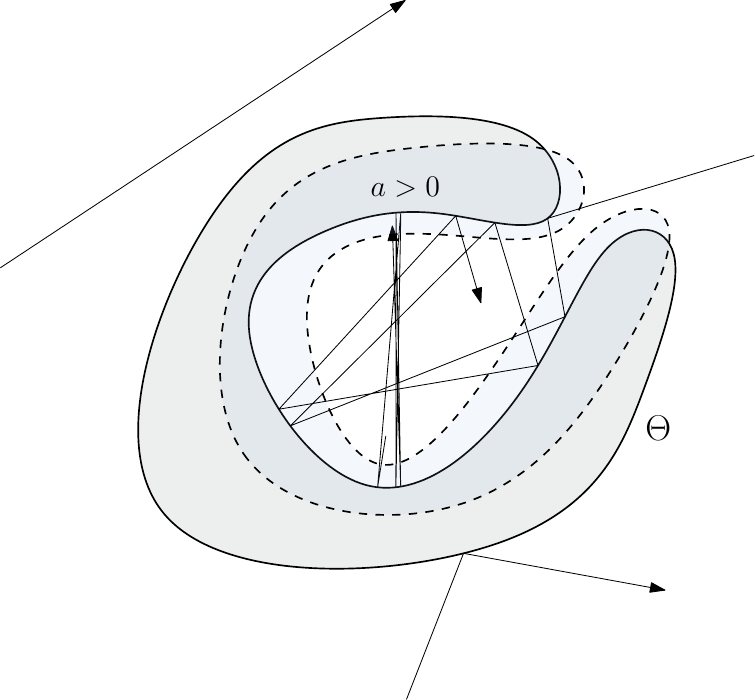}
\caption{Schematic representation of the control condition in the case of an obstacle. Only the non star-shaped portion of the boundary needs to be controlled.} 
\end{center}
\end{figure}

\begin{proof}[Sketch of the proof]
    The analogous to the Morawetz computation presented in the proof of Proposition \ref{prop:Mor} now writes
    \begin{equation}\label{eqVirVir_o}
    \begin{aligned}
    \operatorname{Im}\int  \bar u(t) \nabla  u(t) \cdot \nabla \chi dx &=
    \operatorname{Im} \int (\bar u(s) \nabla  u(s) \cdot \nabla \chi )dx  + \int_s^t \int     2 D^2 \chi \nabla u \cdot \nabla \bar u dx d\tau- \frac 12 \Delta^2 \chi |u|^2\\ &+ \frac 12 \Delta \chi |u|^4 
     - 2 \operatorname{Im} ( a \bar u \nabla u \cdot \nabla \chi) \, dx d\tau + \int_s^t \int_{\partial \Omega} |\partial_nu|^2 \nabla \chi \cdot n \, d\sigma d\tau.
      \end{aligned}
    \end{equation}
    With the above at hand, the supplementary ingredient to the computations of \S\ref{s:nrj} and \S\ref{s:bil} is the following boundary hidden-regularity estimate.
    \begin{lem} \label{lem:hidden_reg}
    Let $\omega \subset \partial \Omega$. Then, for any $K \subset \mathbb R^3$ compact so that $\omega \Subset K$, we have
    \begin{multline*}
            \int_s^t \int_{\omega} |\partial_nu|^2  \, d\sigma d\tau 
\lesssim E[u(t)] + E[u(s)] \\ +  
    \int_s^t \int_{K\cap \Omega} |\nabla u|^2 + |u|^2 + |u|^{4} \,dxd\tau + \int_s^t \int_{\Omega} |a||\nabla u ||u| \, dxd\tau.
    \end{multline*}
    \end{lem}
    \begin{proof}
    Let $Z \in C^\infty_c(\mathbb R^3)$ be so that $\nabla Z(x) = n(x)$ on $\partial \Omega$, and $\psi \in C^\infty_c(\mathbb R^3)$ be so that $\operatorname{supp} \psi \subset K$ and $\psi = 1$ near $\omega$. Applying (\ref{eqVirVir_o}) with $\chi := Z \psi$ gives the result.
    \end{proof}
    We now let $K_0 \subset \mathbb R^3$ so that $\big\{ \partial \Theta \cdot n(x) < 0\big\} \Subset K_0 \Subset \big\{ a>0\big\}$. Combining
    Lemma \ref{lem:hidden_reg}, used with $\omega := \big\{ \partial \Theta \cdot n(x) < 0\big\}$ and $K := K_0$, with (\ref{eqVirVir_o}), gives the following analogous to Proposition \ref{prop:Mor} in the case $\chi(x) = \sqrt{1+x^2}$: 
    \begin{align}\label{eqVirVir_o2}
    &\operatorname{Im}\int  \bar u(t) \nabla  u(t) \cdot \nabla \chi dx \leq
    \operatorname{Im} \int (\bar u(s) \nabla  u(s) \cdot \nabla \chi ) + \int_s^t \int    2 D^2 \chi \nabla u \cdot \nabla \bar u - \frac 12 \Delta^2 \chi |u|^2\\ &\hspace{5cm}+ \frac 12 \Delta \chi |u|^4 
       \, dx d\tau + 
     \mathcal E(t,s), \nonumber \\
     &|\mathcal E(t,s)| \lesssim 
     E[u(t)] + E[u(s)] + \int_s^t \int_\Omega |a||u||\nabla u|\,dxd\tau + \int_s^t \int_{K_0\cap \Omega}  |\nabla u|^2 + |u|^2 + |u|^4\, dx d\tau, \nonumber
      \end{align}
      note in particular that we used the sign of $\nabla \chi \cdot n$ in $\partial \Omega \backslash \omega$.
    Using (\ref{eqVirVir_o2}), one obtains the analogous to Proposition \ref{prpUnifEst} and its consequences, with the same proof, $K_0 \cap \Omega$ playing the role of $\operatorname{supp}(G - \I)$. The analogous to Proposition \ref{prop:L4L4} is then shown similarly by using Lemma \ref{lem:hidden_reg} with $K:= B(0, R)$, $\partial \Omega \subset B(0, R)$, to control the boundary terms arising in the computation. Finally, the end of the proof presented in \S5 holds in the same way, by taking the cut-off $\chi$ equal to one in a neighborhood of $\operatorname{supp}a \cup \partial \Omega$.
    \end{proof}

\section{Mildly trapping case} \label{app:non_trap}

In this appendix, we show scattering up to $H^1$ in the particular case where $G$ induces a mildly trapping Hamiltonian flow. More precisely, we assume that the undamped linear propagator verifies global Strichartz estimates without loss, and a local energy decay estimate:

\begin{assum} \label{ass:mid_trap}
    We assume that global-in-time, non-endpoint Strichartz estimates without loss hold for $e^{it\Delta_G}$: for any $(p, q)$ satisfying condition \eqref{eq:StrichS} with $p>2$, for any $s \geq 0$, $u \in H^s(\R^3)$, and $f \in L^1((0,t), H^s(\R^3))$, the following estimates hold with implicit constants independent of $t$:
    \begin{equation}\label{eqStrichR31}
        \| e^{it\Delta_G} u \|_{L^p((0,t), W^{s,q}(\R^3))} \lesssim \| u \|_{H^s(\R^3)},
    \end{equation}
    \begin{equation}\label{eqStrichR32}
        \left\| \int_0^t e^{i(t-s)\Delta_G} f(s) \, ds \right\|_{L^p((0,t), W^{s,q}(\R^3))} \lesssim \| f \|_{L^1(0,t),H^s(\R^3))}.
    \end{equation}
    In addition, we assume the following local energy decay estimate:
    \begin{equation}\label{eqGenLEC}
        \| e^{it\Delta_G} u \|_{L^2(\mathbb R, L^2(\R ^3))} \lesssim \| u \|_{L^2(\R^3)}.
    \end{equation}
\end{assum}

A primary and significant example of a matrix satisfying the above is the flat case $G = \I$, corresponding to the classical cubic defocusing nonlinear Sch\"odinger equation (NLS) with a linear, spatially localized dissipation term (see \eqref{eq:DNLS}). More generally, any non-trapping perturbation $G$ (in the sense that all the Hamiltonian trajectories defined by (\ref{eq:Ham}) go to infinity) verifies Assumption \ref{ass:mid_trap}, as, as mentionned previously, the non-trapping assumption implies global-in-time local smoothing estimate \cite{BuGeTz04b}, which, together with the frequency-localized Strichartz estimates of \cite{BuGeTz04}, implies global-in-time Strichartz estimates without loss \cite{Iv10, StTa02}. Note that Assumption \ref{ass:mid_trap} is more general and allows some weak trapped trajectories, such as in hyperbolic trapping, for which a smoothing estimate with an arbitrary loss of $\epsilon>0$ derivatives (hence in particular (\ref{eqGenLEC})), and global Strichartz estimates without loss, are expected to hold (see \cite{BuGuHa10} in the case of the Laplace operator on a manifold).

For the remainder of this section, we assume that $G$ verifies Assumption \ref{ass:mid_trap}. Under this assumption;

\begin{enumerate}
    \item The {local well-posedness result} stated in Proposition \ref{prpLocExsCubic} is now valid for initial data \( u_0 \in H^1(\R^3) \), as it follows directly from the usual Kato contraction argument (e.g., see \cite{Ca03}) thanks to the Strichartz estimates without loss of Assumption \ref{ass:mid_trap}. 

    \item The uniform estimates on the energy in Proposition \ref{prpUnifEst}, the $L^4L^4$ control in Proposition \ref{prop:L4L4} and the local energy decay in Corollary \ref{cor:Ubound1} hold with the above proofs, as they are independent of whether $G$ is trapping or not. 
\end{enumerate}
We summarize (1)--(2) in the following Proposition:
\begin{prop}
    Let $G$ verifying Assumption \ref{ass:mid_trap}, and  $a$ verifying the control condition \eqref{eq:control}. Then, for any initial data \( u_0 \in H^1(\R^3) \), there exists a unique global solution \( u \in C([0, \infty); H^1(\R^3)) \) to \eqref{eq:DNLS}. It satisfies the bound:
    \begin{equation}\label{eqH1L4Cont}
        \sup_{t > 0} \| u(t) \|_{H^1(\mathbb R^3)} + \| u \|_{L^4(\R_+, L^4(\R^3))} \lesssim 1.
    \end{equation}
    Moreover, for any $R>0$, there is $C>0$ so that
\begin{equation*}
    \Vert u \Vert_{L^2(\mathbb R_+, H^1(B(0,R)))} \leq C.
\end{equation*}
\end{prop}

In the mildly-trapping case, we establish that it is possible to achieve \textit{scattering in \( H^1(\R^3) \)} rather than in \( H^{1-}(\R^3) \), as stated in Theorem \ref{thm:ScatDef}. While the result in \( H^{1-}(\R^3) \) follows from the arguments presented in Section \ref{secEnd}, 
the mildly trapping condition enables a stronger result. Specifically, 
the availability of global-in-time Strichartz estimates without loss for the undamped flow $e^{it\Delta_G}$ permits to show scattering in $H^1(\R ^3)$ from the control of a global Strichartz norm, and we obtain:

\begin{thm} \label{th:mid_trapp}
    Let $G$ verifies Assumption \ref{ass:mid_trap}, and let $a$ satisfy condition \eqref{eq:control}. Then, for any initial data \( u_0 \in H^1(\R^3) \), there exists $u_+, \widetilde u_+ \in H^1(\R^3)$ so that the solution $u \in C([0, +\infty), H^1(\R ^3))$ to (\ref{eq:DNLS}) with data $u_0$ verifies
    $$ 
\lim_{t \to +\infty} \Vert u(t) - e^{it\Delta} \widetilde u_+ \Vert_{H^1(\R ^3)} =
\lim_{t \to +\infty} \Vert u(t) - e^{it\Delta_G}  u_+ \Vert_{H^1(\R ^3)} = 0.
    $$
\end{thm}

To control a global Strichartz norm at the $H^1(\R ^3)$ level from the $L^4 L^4$ global control, we will argue similarly as in \cite{CoKeStTaTa04}. The supplementary ingredient will be to treat $iau$ as a source term. We will be able to do so thanks to the following:  

\begin{lem}
    Let $G$ verifies Assumption \ref{ass:mid_trap}, and let $a$ satisfy condition \eqref{eq:control}. Then, for any $(p,q)$ Strichartz-admissible with $p>2$ and any solution to (\ref{eq:DNLS}), the following holds:
    \begin{equation}\label{eqB}
         \left\| \int_0^t e^{i(t-s)\Delta_G} a u \, ds \right\|_{L_t^p(\mathbb R_+, W^{1,q}(\R^3))} \lesssim \| u_0 \|_{H^1(\mathbb R^3)} + \Vert u_0 \Vert^2_{L^4(\R^3)}.
    \end{equation}
\end{lem}

\begin{proof}
We will show that for any $\chi \in C^\infty_c(\mathbb R^3)$ the operator
$$
T : f \in L^2 (\mathbb R, H^{1}(\R^3)) \mapsto 
\int_{s<t} e^{i(t-s)\Delta_G} \chi f(s) \, ds \in L^p(\mathbb R, W^{1,q}(\R^3))
$$
is bounded. 

Indeed, if it is the case, take $\chi \in C^\infty_c(\R^3)$ be a smooth cutoff function with $\chi = 1$ on $\operatorname{supp} a$.
    Using the operator bound on $f(s) := au(s) \mathbf 1_{s \geq 0}$, then the local energy decay property \eqref{eq:L2H1LocDec}, we obtain
    \begin{equation*}
         \left\| \int_0^t e^{i(t-s)\Delta_G} \chi a u \, ds \right\|^2_{L_t^p(\mathbb R_+, W^{1,q}(\R^3))} \lesssim \| a u \|^2_{L^2((0, \infty), H^{1}(\R^3))} \lesssim \| u_0 \|^2_{H^1(\mathbb R^3)} + \| u_0 \|^4_{L^4(\R^3)},
    \end{equation*}
    where the compact support of $a$ ensures the bound, and we are done. 
    
    To this end, we first claim that, for any $s\in[0,2]$
\begin{equation} \label{eq:smooth_bdd}
        f \in L^2 (\mathbb R, H^{s}(\R^3)) \mapsto 
        \int_{\mathbb R} e^{-is\Delta_G} \chi f(s) \, ds \in H^s(\R^3)
\end{equation}
    is bounded.
    Indeed, for $s = 0$, the claim is the dual estimate to local energy decay (\ref{eqGenLEC}). Next, for $s=2$, we write, letting $\widetilde \chi \in C^\infty_c(\R^3)$ be equal to one on $\operatorname{supp}\chi$ and using the claim with $s=0$
    \begin{align*}
    \Big\Vert \Delta_G \int_{\mathbb R} e^{-is\Delta_G} \chi f(s) \, ds\Big\Vert_{L^2(\R^3)} 
    &= \Big\Vert \int_{\mathbb R} e^{-is\Delta_G} \widetilde \chi \Delta_G(\chi f(s)) \, ds\Big\Vert_{L^2(\R^3)} \\
    &\lesssim \Vert \Delta_G(\chi f(s)) \Vert_{L^2(\R, L^2(\R^3))} \lesssim \Vert f \Vert_{L^2(\R, H^{2}(\R^3))}.
    \end{align*}
    But, by elliptic regularity
    $$
    \Big\Vert \int_{\mathbb R} e^{-is\Delta_G} \chi f(s) \, ds\Big\Vert_{H^2(\R^3)} 
 \lesssim\Big\Vert \Delta_G \int_{\mathbb R} e^{-is\Delta_G} \chi f(s) \, ds\Big\Vert_{L^2(\R^3)} + 
 \Big\Vert \int_{\mathbb R} e^{-is\Delta_G} \chi f(s) \, ds\Big\Vert_{L^2(\R^3)}, 
    $$
    hence the claim for $s=2$ follows. The claim is shown for any $s\in[0,2]$ by interpolation.
    
    To conclude, we use
    a Theorem due to Christ and Kiselev \cite{ChKi01}, in the following form from \cite[Lemma 5.6]{PV1}: 
    \begin{lem} \label{lem:CK}
    Let $U(t)$ be a one-parameter group of operators, $1\leq a < b \leq +\infty$, $H$ a Hilbert space and $B$, $\widetilde B$ two Banach spaces.   Assume that
    $$
    \varphi \in H \mapsto U(t) \varphi \in L^b(B), \hspace{0.5cm} g \in L^a(\widetilde B) \mapsto \int U(-s)g(s) ds \in H
    $$
    are bounded. Then, the operator
    $$
    g \in L^a(\widetilde B) \mapsto \int_{s<t} U(t-s)g(s) ds \in L^b(B)
    $$
    is bounded.
    \end{lem}
    Take $U(t) = e^{it\Delta_G}$, $H = H^1(\mathbb R^3)$, $a=2$, $b = p$, $B=W^{1,q}$, $\widetilde B = H^{1}$. The claim with $s=1$ together with the Strichartz estimates (\ref{eqStrichR31}) and Lemma \ref{lem:CK} show that $T$ is bounded, and the proof is completed.
\end{proof}

\begin{proof}[Proof of Theorem \ref{th:mid_trapp}]
Scattering in \( H^1(\R^3) \) will follow once we establish a uniform bound of the form
\begin{equation}\label{eqZ1}
    Z(t) := \sup_{\substack{p, q \text{ admissible}\\ p>2}} \| u \|_{L^p((0,t), W^{1,q}(\R^3))} \leq C(\| u_0 \|_{H^1(\mathbb R^3)}).
\end{equation}

From the $L^4L^4$ control \eqref{eqH1L4Cont}, we partition $[0, \infty)$ into a finite number of intervals \( J_1, \dots, J_K \) such that for each \( i = 1, \dots, K \),
\begin{equation*}
    \| u \|_{L^4(J_i, L^4(\R^3))} \leq \epsilon,
\end{equation*}
where \(\epsilon = \epsilon(\| u_0 \|_{H^1(\mathbb R^3)})\) is a small constant to be determined. For any \( t \in J_i \), the Strichartz estimates \eqref{eq:StrichS} together with (\ref{eqB}) yield
\begin{equation*}
    Z(t) \lesssim  \| |u|^2 u \|_{L^{10/7}((0,t), W^{1,10/7}(\R^3))} + \Vert u_0 \Vert_{H^1(\mathbb R^3)} + \Vert u_0 \Vert_{L^4(\R^3)}^2.
\end{equation*}
Using the fractional Leibniz rule, we have
\begin{equation*}
\begin{aligned}
     \| |u|^2 u \|_{L^{10/7}((0,t), W^{1,10/7}(\R^3))} & \lesssim \| u \|_{L^{10/3}((0,t), W^{1,10/3})} \| u \|_{L^5((0,t), L^5(\R^3))}^2 \\ 
     &\lesssim Z(t) \| u \|_{L^4((0,t), L^4(\R^3))}^\alpha \| u \|_{L^6((0,t), L^6(\R^3)}^\beta
\end{aligned}
\end{equation*}
for some \(\alpha, \beta > 0\), where we interpolate \( L^5L^5 \) between \( L^4L^4 \) and \( L^6L^6 \). By Sobolev embedding, we bound the \( L^6L^6 \)-norm as:
\begin{equation*}
    \| u \|_{L^6((0,t), L^6(\R^3))} \lesssim \| u \|_{L^6((0,t), W^{2/3,18/7}(\R^3))} \leq Z(t),
\end{equation*}
and we conclude
\begin{equation*}
    Z(t) \lesssim \| u_0 \|_{H^1(\mathbb R^3)} + \| u_0 \|_{L^4}^2+ \epsilon^\alpha Z(t)^{1+\delta_2},
\end{equation*}
for some constants \(\alpha, \delta_2 > 0\). Choosing \(\epsilon > 0\) sufficiently small ensures that \eqref{eqZ1} holds on \( J_i \). We establish the bound globally by repeating this argument over all intervals \( J_1, \dots, J_K \). 

Scattering in \( H^1(\R^3) \) to a wave $e^{it\Delta_G}u_+$ now follows from the similar arguments as used in the conclusion of the proof of Theorem \ref{thm:ScatDef} in Section \ref{secEnd}, replacing the \( H^{\frac 12} \)-norm with the \( H^1 \)-norm, and without cutting away $\operatorname{supp} a$. More precisely, from the Duhamel formula, it suffices to show that
    \begin{equation} \label{eq:scat_crit_I1NT}
            \Big\Vert  \int_t^\infty e^{-is\Delta_G}u|u|^{2}(s) \, ds \Big\Vert_{H^{1}(\mathbb R^3)} \to 0 \hspace{0.3cm}\text{ as }\hspace{0.3cm}t\to +\infty,
    \end{equation}
    and
    \begin{equation}\label{eq:scat_crit_I2NT}
            \Big\Vert  \int_t^\infty e^{-is\Delta_G} a u(s) ds \Big\Vert_{H^{1}(\mathbb R^3)} \to 0 \hspace{0.3cm}\text{ as }\hspace{0.3cm}t\to +\infty.
    \end{equation}
    For \eqref{eq:scat_crit_I1NT} we use \eqref{eqStrichR32}, H\"older's inequality and Riesz-Thorin theorem to get 
    \begin{align*}
        \Big\Vert  \int_t^\infty e^{-is\Delta_G} |u|^2 u(s) ds \Big\Vert_{H^{1}(\mathbb R^3)} &\lesssim \| |u|^2 u \|_{L^\frac{8}{5}((t,\infty), W^{\frac{4}{3},1}(\R^3))} \lesssim \| u \|_{L^8((t,\infty),L^4(\R^3))}^2 \| u \|_{L^\frac{8}{3}((t,\infty), W^{4,1}(\R^3))} \\
        &\lesssim \| u \|_{L^\infty((t,\infty),L^4)} \| u \|_{L^4((t,\infty),L^4(\R^3))}   \| u \|_{L^\frac{8}{3}((t,\infty), W^{4,1}(\R^3))} \to 0
    \end{align*}
    since $(\frac{8}{3},4)$ is admissible. 
    For (\ref{eq:scat_crit_I2NT}), we use the operator bound (\ref{eq:smooth_bdd}) for $s=1$ and $\chi \in C^\infty_c(\mathbb R^3)$ so that $\chi = 1$ on $\operatorname{supp} a$, on $f(s) := au(s) \mathbf 1_{s \in [t, +\infty)}$ and obtain
    $$
    \Big\Vert  \int_t^\infty e^{-is\Delta_G} a u(s) ds \Big\Vert_{H^{1}(\mathbb R^3)} \lesssim \Vert au \Vert_{L^2([t, + \infty), H^{1}(\R^3))} \to 0
    $$
    thanks to \eqref{eq:L2H1LocDec}. 
    Finally, to show that scattering also holds to a free wave $e^{it\Delta}\widetilde u_+$, we use the following linear scattering result:
    \begin{lem}
        For any $u_+ \in H^1(\mathbb R_+)$, there exists $\widetilde u_+$ so that, as $t \to +\infty$
        $$
\Vert e^{it\Delta}\widetilde u_+ - e^{it\Delta_G} u_+ \Vert_{H^1(\mathbb R^3)} \to 0.
        $$
    \end{lem}
    \begin{proof}
    Let $\chi \in C^\infty(\mathbb R^3)$ so that $\chi = 1$ on $\operatorname{supp}(G - \I)$.      First, observe that $\chi e^{it\Delta_G} u_+ \to 0$ in $H^1(\mathbb R^3)$. Indeed, one can show from (\ref{eqGenLEC}), similarly to the proof of Corollary \ref{cor:Ubound2}, that $\chi e^{it\Delta_G} u_+ \to 0$ in $L^2(\mathbb R^3)$. This implies the claim by interpolation for $u_+ \in H^s(\mathbb R^3)$ with $s>1$, and the claim then follows for any $u_+\in H^1(\R ^3)$ by approximation. Hence, it suffices to show the existence of $\widetilde u_+ \in H^1(\mathbb R^3)$ so that
    $$
 \Vert e^{it\Delta}\widetilde u_+ - (1-\chi)e^{it\Delta_G} u_+ \Vert_{H^1(\mathbb R^3)} \to 0.
     $$
    Let $v(t) := (1-\chi)e^{it\Delta_G} u_+$. It verifies
    $$
i\partial_t v + \Delta v = [\Delta, \chi] u,
    $$
    hence from Duhamel's formula
    $$
    v(t) = e^{it\Delta}\Big( (1-\chi) u_+ - i\int_0^t e^{-is\Delta}  [\Delta, \chi] u(s) ds  \Big).
    $$
    Observe that the resolvent estimate without loss sufficiently far from the origin of Cardoso-Vodev \cite{CaVo02} implies with the arguments of \cite[Proposition 2.7]{BuGeTz04b}, the following smoothing estimate without loss far away from $\operatorname{supp}(G-I)$ for $e^{it\Delta_G}$: there exists $R\gg 1$ big enough so that for any $\psi \in C^\infty_c(\mathbb R^3)$ supported away from $B(0,R)$ and any $u_0 \in L^2$
    \begin{equation}    \label{eq:smooth_away}
    \Vert \psi e^{it\Delta_G} u_0 \Vert_{L^2(\mathbb R, H^{\frac 12}(\R ^3))}
    \lesssim \Vert u_0 \Vert_{L^2}.
    \end{equation}
    Taking $\chi =1$ on $B(0,R)$, from the dual estimate to the smoothing effect for $e^{-is\Delta}$ together with (\ref{eq:smooth_away}), we obtain as $[\Delta, \chi]$ is supported away from $B(0,R)$:
    $$
    \Big\Vert \int_t^\infty e^{-is\Delta}  [\Delta, \chi] u(s) ds \Big\Vert_{H^1(\R^3) }\to 0,
    $$
    hence the result follows by setting
    $$
    \widetilde u_+ := (1-\chi) u_+ - i\int_0^\infty e^{-is\Delta}  [\Delta, \chi] u(s) ds.
    $$
    \end{proof}
    This concludes the proof of Theorem \ref{th:mid_trapp}.
    \end{proof}
    \subsection*{Acknowledgments} We are grateful to professor Nicola Visciglia for highlighting that we could improve our initial result of scattering in $H^{\frac 12}$ to $H^{1-}$ by interpolation.
\bibliographystyle{abbrv}
\bibliography{biblio}
\end{document}